\documentclass[10pt]{amsart}
\usepackage{enumerate}
\usepackage{comment}

\makeatletter

\@namedef{subjclassname@2010}{

  \textup{2010} Mathematics Subject Classification}

\usepackage{amssymb, amsmath,amsthm}
\usepackage{mathrsfs}
\newtheorem{thm}{Theorem}[section]

\newtheorem{prop}[thm]{Proposition}

\newtheorem{cor}[thm]{Corollary}
\newtheorem{lem}[thm]{Lemma}
\theoremstyle{definition}
\newtheorem{rem}[thm]{Remark}

\newcommand{\ra}{\rightarrow}
\newcommand{\bk}{\backslash}
\newcommand{\mc}{\mathcal}
\newcommand{\mf}{\mathfrak}
\newcommand{\mb}{\mathbb}
\newcommand{\sg}{\sigma}

\renewcommand{\ss}{\substack}

\newcommand{\e}{\varepsilon}

\renewcommand{\bar}{\overline}

\title[Sign Changes and Norm Forms]{Sign changes of Fourier coefficients of cusp forms at norm form arguments}
\author{Alexander P. Mangerel}
\address{Department of Mathematical Sciences, Durham University, Stockton Road, Durham, DH1 3LE, UK}
\email{smangerel@gmail.com}

\begin{document}

\begin{abstract}
Let $f$ be a non-CM Hecke eigencusp form of level 1 and fixed weight, and let $\{\lambda_f(n)\}_n$ be its sequence of normalized Fourier coefficients. We show that if $K/ \mb{Q}$ is any number field, and $\mc{N}_K$ denotes the collection of integers representable as norms of integral ideals of $K$, then a positive proportion of the positive integers $n \in \mc{N}_K$ yield a sign change for the sequence $\{\lambda_f(n)\}_{n \in \mc{N}_K}$. More precisely, for a positive proportion of $n \in \mc{N}_K \cap [1,X]$ we have $\lambda_f(n)\lambda_f(n') < 0$ where $n'$ is the first element of $\mc{N}_K$ greater than $n$ for which $\lambda_f(n') \neq 0$. \\
For example, for $K = \mb{Q}(i)$ and $\mc{N}_K = \{m^2+n^2 : m,n \in \mb{Z}\}$ the set of sums of two squares, we obtain $\gg_f X/\sqrt{\log X}$ such sign changes, which is best possible (up to the implicit constant) and improves upon work of Banerjee and Pandey. Our proof relies on recent work of Matom\"{a}ki and Radziwi\l\l{} on sparsely-supported multiplicative functions, together with some technical refinements of their results due to the author. \\
In a related vein, we also consider the question of sign changes along shifted sums of two squares, for which multiplicative techniques do not directly apply. Using estimates for shifted convolution sums among other techniques, we establish that for any fixed $a \neq 0$ there are $\gg_{f,\e} X^{1/2-\e}$ sign changes for $\lambda_f$ along the sequence of integers of the form $a + m^2 + n^2 \leq X$.
%we further show that there are $\gg X^{1/6-\e}$ \emph{consecutive} $n,n+1 \in \mc{N} \cap [1,X]$ for which $\lambda_f(n)\lambda_f(n+1) < 0$. The proof is based on a treatment of shifted convolution sums both for $\lambda_f(n)$ as well as $r(n)$, the number of representations of $n$ as a sum of two squares.
\end{abstract}

\maketitle

\section{Introduction}
The study of sign patterns of real-valued multiplicative functions at consecutive integers
%The subject of sign changes of multiplicative functions 
has received a lot of attention in recent years, as a means of investigating the apparently random behaviour of multiplicative functions in interactions with additive patterns of integers. 
Of classical, as well as modern, interest is the particular collection of multiplicative functions that arise from the sequence of Fourier coefficients of a normalized Hecke eigencusp form $f$ of some weight and level. Here, we study the case of forms with level $1$ and without complex multiplication. In the sequel, for such a cusp form $f$ of weight $k$ and level 1 we write
$$
f(z) := \sum_{n \geq 1} \lambda_f(n)n^{\tfrac{k-1}{2}} e(nz), \quad \text{Im}(z) > 0,
$$
where, as usual $e(z) := e^{2\pi i z}$ for $z \in \mb{C}$. In this case the sequence $\{\lambda_f(n)\}_n$ is real, and it is natural to study its \emph{sign changes}, i.e., the set of $n$ for which $\lambda_f(n)\lambda_f(n+1) < 0$. \\
A number of works (see e.g., \cite{KLSW}, \cite{Lam}) have explored the relationship between the distribution of signs and sign changes of Fourier coefficients of cusp forms and the least quadratic non-residue problem modulo primes $p$. Moreover, in the spirit of exploring the mass distribution of Hecke-Maass cusp forms in the weight aspect, Ghosh and Sarnak \cite{GhSa} exhibited a relationship between the distribution of ``real zeros'' of such forms and counts for sign changes of their coefficients.\\
Improving on work of Lau and Wu \cite{LauWu}, Matom\"{a}ki and Radziwi\l\l{} \cite{MRCusp} obtained the optimal result that there are $\gg_f X$ sign changes for the sequence $\{\lambda_f(n)\}_{n\leq X}$. 
Shortly thereafter, they managed to reprove this result in a much more general context introduced in the breakthrough work \cite{MR} on general bounded multiplicative functions $g$. This new proof is based on relating short and long partial sums of $\text{sign}(g(n))$ (using the convention $\text{sign}(0) = 0$), discovering that sufficient cancellation in the long partial sums of $g$ imply that in \emph{typical} short intervals of bounded but large length $g$ must change sign at least once. \\
More generally, given a set $\mc{S} \subset \mb{N}$ and a map $g: \mb{N} \ra \mb{R}$, by a \emph{sign change of $g$ on $\mc{S}$} we mean a pair of elements $n < n'$ of $\mc{S}$ such that
%\footnote{Note, once and for all, that we do not account for zero-values of $g$ in counting sign changes.} 
$g(n)g(n') < 0$, and such that any $n < m < n'$ belonging to $\mc{S}$ satisfies $g(m) = 0$. 
One can naturally ask what, if anything, can be said about the sign changes that arise along \emph{sparse} subsequences of positive integers of arithmetic interest. As an example, in \cite[Thm. 5]{Mur}, M.R. Murty unconditionally estimates the number of sign changes among the prime values $\{\lambda_f(p)\}_p$. Using a Hoheisel-type argument, based on zero-density estimates for the corresponding Hecke $L$-function, he obtains cancellation in the short interval prime sums
$$
\sum_{x < n \leq x+h} \lambda_f(p)\log p,
$$
for $h < x^{1-\delta}$ and some small $\delta = \delta(f) > 0$, enabling him to obtain $\gg_f x^{\delta}$ such sign changes among prime values. We might expect this to be rather far from the truth, and that perhaps there should even be $\gg_f \pi(x)$ such sign changes $p \leq x$. \\
In this paper, instead of the sequence of primes we will consider the model set of sums of two squares,
$$
\mc{N} := \{a^2 + b^2 : a,b \in \mb{Z}\},
% \cap \mb{N},
$$ 
as a setting in which strong lower bounds on the number of sign changes can be obtained. This set also reveals itself as a natural choice given the relevance of the generating function for the set of perfect squares, the Jacobi theta function, in the theory of modular forms. \\
In \cite{BanPa}, the authors consider the problem of counting sign changes in the sequence $\{\lambda_f(n)\}_{n \in \mc{N}}$. They showed that there are $\gg x^{1/8-\e}$ such sign changes in $(x,2x]$. 
Their proof ultimately relies on comparing the partial sums 
$$
\sum_{x < n \leq x+h} \lambda_f(n)r(n) \quad \text{ and } \sum_{x < n \leq x+h} \lambda_f(n)^2 r(n),
$$ 
where $r(n)$ denotes the number of representations of $n$ as a sum of two squares, and with $h = h(x) \geq x^{7/8+\e}$. Using a contour integration argument involving Rankin-Selberg $L$-functions, they derive a contradiction to the assertion that $\lambda_f(n) \geq 0$ (say) for all $x < n \leq x+h$. The structure of $\mc{N}$ (in particular the convolution formula $\tfrac{1}{4}r = 1\ast \chi_{4}$, where $\chi_{4}$ is the non-principal character modulo $4$) seems to play a crucial role. \\
%By a contour integration argument involving the Rankin-Selberg $L$-functions $L(s, f \otimes \theta^2)$ and $L(s, f \otimes f \otimes \chi_{-4})$ (using $\frac{1}{4} r_2(n) = 1 \ast \chi_{-4}$), they show that the former partial sum exhibits power cancellation which allows them to take $h = x^{1-1/8 + \delta}$ in an asymptotic formula. Their technique is inherently lossy in that they use a divisor-bound $|\lambda_f(n)| \leq d(n) \ll x^{\e}$ to make the pass 
%$$
%\sum_{x < n \leq x+h}\lambda_f(n)^2 r_2(n) \ll_{\e} x^{\e} \sum_{x < n \leq x+ h} \lambda_f(n)r_2(n),
%$$ 
%so they could not, for instance conclude any sort of result of the shape $X/(\log X)^A$ (i.e., they could never take $h$ a fixed power of $\log x$, say). \\
In view of the result of Matom\"{a}ki and Radziwi\l\l{} mentioned above, even accounting for the sparseness of $\mc{N}$ it seems that one ought to do better, with the optimal result expected to be $\gg \tfrac{x}{\sqrt{\log x}} \asymp |\mc{N} \cap [1,x]|$ sign changes in $[1,x]$. 
%the order of magnitude of the number of elements in the set. 
It should be noted that this cannot be achieved by their method, which relies on 
%comparing summands $\lambda_f(n)^2r_2(n)$ and $\lambda_f(n)r_2(n)$ on short intervals $[x,x+h]$ where $\lambda_f(n)$ has constant sign, using the 
Deligne's bound $|\lambda_f(n)| \leq d(n)$ as
$$
\sum_{x < n \leq x+h}\lambda_f(n)^2 r(n) \leq \left(\max_{x < m \leq x+h} d(m)\right) \sum_{x < n \leq x+ h} \lambda_f(n)r(n)
$$ 
to obtain a contradiction to the purported non-negativity of $\lambda_f(n)$ on $[x,x+h]$. This forces any admissible choice of length $h = h(x)$ to satisfy $h \geq \exp\left((1+o(1))\frac{\log x}{\log\log x}\right)$ by invoking pointwise bounds on the divisor function, and the number of sign changes obtained with this argument (as discussed later) is $\gg x/h$.  \\
In this paper, we improve upon the main theorem in \cite{BanPa} by making extensive use of the techniques in the more recent paper \cite{MRII} of Matom\"{a}ki and Radziwi\l\l{}, which among other things are applicable to sparsely-supported multiplicative functions (of which $f(n)1_{\mc{N}}(n)$ is an example). By additionally incorporating some refinements to these methods from the author's paper \cite{MRDB} (which render their results slightly more amenable to the study Fourier coefficients of cusp forms), we in fact obtain the optimal result.
\begin{cor}\label{cor:SOTS}
Let $f$ be a primitive Hecke eigencusp form without complex multiplication of weight $k\geq 2$ for the full modular group, and let $X$ be large. 
Then $\mc{N} \cap [1,X]$ contains $\gg_f \frac{X}{\sqrt{\log X}}$ sign changes for $\lambda_f$. 
%$\{\lambda_f(n)\}_{\ss{n \in \mc{N} \\ n \leq X}}$ 
%More precisely, there are $\gg X/\sqrt{\log X}$ integers $n \leq X$ such that $n \in \mc{N}$ and if $n' > n$ is the next element in $\mc{N}$ then $\lambda_f(n)\lambda_f(n') < 0$.
\end{cor}
We will actually prove a more general result that provides the optimal number of sign changes in $\{\lambda_f(n)\}_{\ss{n \in \mc{N}_K \\ n \leq X}}$, where $\mc{N}_K$ is the sequence of positive integers arising as norms of algebraic integers in a number field $K/\mb{Q}$. Following Matom\"{a}ki and Radziwi\l\l{} \cite[Sec. 1.3]{MRII}, we refer to these as \emph{norm forms}.  An integer $n$ is thus a norm form of $K$ if there is an algebraic integer $x$ in the ring of integers $\mc{O}_K$ of $K$ such that $N_K(x) = n$, where $N_K$ is the norm map on $K$. As the norm on $K = \mb{Q}(i)$ is simply $N_{\mb{Q}(i)}(a+ib) = a^2+b^2$, we have in the above notation $\mc{N} = \mc{N}_{\mb{Q}(i)}$. Define now
$$
\delta_K(X) := \prod_{\ss{p \leq X \\ p \neq N_K(\mf{a}) \text{ for } \mf{a} \subset \mc{O}_K }} \left(1-\frac{1}{p}\right).
$$
It is known (see e.g., \cite{Odo}) that $|\mc{N}_K \cap [1,X]| \asymp_K X \delta_K(X)$. We prove the following.
\begin{thm}\label{thm:CFNormForms}
Let $K/\mb{Q}$ be a number field, and let $f$ be a primitive Hecke eigencusp form without complex multiplication of weight $k \geq 2$ for the full modular group. As $X\ra \infty$, the number of sign changes of $\lambda_f$ in $\mc{N}_K \cap [1,X]$ is $\gg_{K,f} X\delta_K(X)$. 
\end{thm}

%Returning to the setting of sums of two squares, it is natural to ask about the set of sign changes $\lambda_f(n)\lambda_f(n') < 0$ with $n,n'$ consecutive elements of $\mc{N}$, that can arise when\footnote{We would like to warmly thank Oleksiy Klurman for suggesting this problem.} $n' = n+1$; that is, how often does a sign change arise not only from consecutive sums of two squares, but such that these sums of two squares are also consecutive integers. We obtain a lower bound for the count of such sign changes that also improves upon the bounds in \cite{BanPa}. 
Returning to the setting of sums of two squares, we may modify the problem slightly by asking about sign changes among other patterns of integers, such as shifted sums $a + \mc{N} = \{a + m^2 + n^2 : m,n \in \mb{N}\}$, for $a \in \mb{Z} \bk \{0\}$. In this case, it is more challenging to directly apply tools from multiplicative number theory. Nevertheless, using shifted convolution sum estimates among other techniques, we obtain a lower bound on the number of sign changes of $\lambda_f$ along $a+\mc{N}$.

\begin{thm} \label{thm:consecSoTS}
Fix $a \neq 0$. Then for any $\e > 0$ and $X$ sufficiently large there are $\gg_{\e} X^{1/2-\e}$ sign changes for $\lambda_f$ in $(a+\mc{N}) \cap [1,X]$.
%, with the lone restriction that if $\lambda_f(2^\nu) = 0$ for some $\nu \geq 1$ then $a$ is odd
% for which $n,n+1 \in \mc{N}$.
\end{thm}
%Since, e.g., $\tau(2^\nu) \neq 0$ for all $\nu$, this gives that $\tau(n)$ changes sign $\gg X^{1/2-\e}$ times along integers of the form $\square + \square + a \in [1,X]$.
\subsection{Proof Ideas}
\subsubsection{Proof of Theorem \ref{thm:CFNormForms}}
Let $f$ be a non-CM eigencusp form for $\text{SL}_2(\mb{Z})$, and let $K/ \mb{Q}$ be a number field. Denote by $\sg_f(n)$ the sign of $\lambda_f(n)$, using the convention $\sg_f(n) = 0$ whenever $\lambda_f(n) = 0$. \\
Let $1 \leq h \leq X$. The proof of Theorem \ref{thm:CFNormForms} follows the strategy of \cite[Cor. 3]{MR}, the objective of which is to show that for all but $o(X)$ points $x \in [X,2X]$, for $X$ large, the two averages
$$
\frac{1}{h}\sum_{x < n \leq x + h} 1_{\mc{N}_K}(n) (|\sg_f(n)| + \sg_f(n)) , \quad \quad \frac{1}{h}\sum_{x < n \leq x+h} 1_{\mc{N}_K}(n)(|\sg_f(n)| - \sg_f(n))
$$
are simultaneously $ > 0$. The positivity of the left-hand sum implies the existence of $x < n_1 \leq x+h$ such that $\sg_f(n_1) > 0$, and similarly that of the right-hand sum implies that $\sg_f(n_2) < 0$ for some $x < n_2 \leq x+h$. This gives rise to a sign change in most short intervals $[x,x+h]$, and by dissecting $[X,2X]$ into disjoint such short intervals, yields  $\gg X/h$ distinct sign changes. \\
Of course, as $\mc{N}_K$ is a sparse subset of positive integers containing $\asymp \delta_K(X) X$ integers $n \in [X,2X]$, it is not even guaranteed that the above sums have non-empty support unless $h$ is sufficiently large. To avoid this issue we require, in particular, that $h \geq C \delta_K(X)^{-1}$, for $C > 0$ a large constant (depending at most on $K$ and $f$). \\
Depending on the class number of $K$, the argument must be modified. Consider first when $K$ has class number 1 (this being in particular the case for $K = \mb{Q}(i)$). Then (by Dedekind's ideal factorization theorem) the indicator function $1_{\mc{N}_K}$ is multiplicative.  
%and the above comparison reduces the problem 
Now, naturally if we had a means of making the comparison
$$
\frac{1}{h}\sum_{x < n \leq x+h} 1_{\mc{N}_K}(n) (|\sg_f(n)| + \eta \sg_f(n)) = \frac{1}{X}\sum_{X  <n \leq 2X} 1_{\mc{N}_K}(n)(|\sg_f(n)| + \eta \sg_f(n)) + o(\delta_K(X)),
$$
for \emph{typical} $x \in [X,2X]$ and each $\eta \in \{-1,+1\}$ then our problem becomes substantially easier. The key issue, that of the sparseness of support, is a main obstacle in this comparison. Fortunately, the recent work of Matom\"{a}ki and Radziwi\l\l{} \cite{MRII} is dedicated to addressing exactly such complications, and may be employed to give such a comparison theorem. Our particular application of their methods is worked out in Section \ref{sec:MainArg}.\\
%arguments about long interval averages is sensitive to the class number of the field $K$. 
The problem is thus reduced to to one involving mean values of multiplicative functions with sparse support. As we show using work of Wirsing \cite{Wir} and of Tenenbaum \cite{TenVM}, respectively, we may deduce that
$$
\frac{1}{X}\sum_{X < n \leq 2X} |\sg_f(n)| 1_{\mc{N}_K}(n) \asymp \delta_K(X), \quad \quad \frac{1}{X}\sum_{X < n \leq 2X}  \sg_f(n)1_{\mc{N}_K}(n) = o(\delta_K(X)).
$$
Whereas the former bound is of a more classical nature, the latter relies on an understanding (roughly speaking) of the prime sums
$$
\sum_{\substack{p \leq X \\ p \in \mc{I}_K}} \frac{1-\text{Re}(\sg_f(p)p^{-it})}{p}, \quad |t| \leq \log X,
$$
where $\mc{I}_K$ is a class of ideals of the ring $\mc{O}_K$ of algebraic integers of $K$.  
As a concrete example, when $K = \mb{Q}(i)$ and $\mc{N}_K$ is the sequence of sums of two squares, $\mc{I}_K$ is the set of primes $p \equiv 1 \pmod{4}$ (the collection of primes where $1_{\mc{N}}$ is supported). By employing a (generalization of a) hybrid Chebotarev-Sato-Tate type estimate due to R.M. Murty and V.K. Murty \cite{Mur}, we are able to condition on both the sign of $\lambda_f(p)$ as well as the ideal class of $p$ in order to prove that these sums tend to $\infty$ uniformly in $|t| \leq \log X$ as $X \ra \infty$.  \\
When $K$ has class number $> 1$ the problem is rendered more complicated by the fact that the indicator $1_{\mc{N}_K}$ is no longer a multiplicative function. Luckily, a result of Odoni (discussed in some detail in \cite{MRII}) allows one to express $1_{\mc{N}_K}$ as a linear combination of structured multiplicative functions, and (with some work) similar techniques may then be applied to the individual terms of these linear combinations.
\subsubsection{Proof of Theorem \ref{thm:consecSoTS}}
Fix a non-zero integer $a$. To prove Theorem \ref{thm:consecSoTS}, we make use of more classical arguments about sign changes. Namely, we find a lower bound on the least $h$ such that $\lambda_f(n)$ must change sign in an interval $[x,x+h]$ for \emph{typical} $x \in [X,2X]$. Dually, we show by way of contradiction that if $\lambda_f(n) \geq 0$, say, for all $x < n \leq x+h$ then we obtain contradictory upper and lower bounds for the quantity
$$
\sum_{x < n \leq x+h} \lambda_f(n) r(n-a)
$$
for typical $x \in [X,2X]$, provided $h \gg X^{1/2+\e}$. This results in $\gg X/h \gg X^{1/2-\e}$ distinct sign changes. Here, as above, $r(n)$ denotes the number of representations of $n$ as a sum of two squares. \\
The upper bound we need is furnished by estimates for shifted convolution sums arising from the spectral theory of automorphic forms. In this context, one obtains the square-root cancelling bound $O_f(X^{1/2+\e})$ using the work of Ravindran \cite{Rav}. \\
The lower bound requires more work, and principally involves restricting the sum to those $n$ for which $|\lambda_f(n)| > X^{-\delta}$. Assuming $\lambda_f(n) \geq 0$ for all $x < n \leq x+h$, it follows that then 
\begin{equation}\label{eq:conditioning}
\sum_{x < n \leq x+h} \lambda_f(n)r(n-a) \geq X^{-\delta} \sum_{\ss{x < n \leq x+h \\ \lambda_f(n) > X^{-\delta}}} r(n-a) = X^{-\delta}\left(\sum_{\ss{x < n \leq x+h \\ \lambda_f(n) \neq 0}} r(n-a) - \sum_{\ss{x < n \leq x+h \\ 0 < \lambda_f(n) \leq X^{-\delta}}} r(n-a)\right).
\end{equation}
Since, as Serre \cite{Ser} showed, the set of primes $p$ such that $\lambda_f(p) = 0$ is quite sparse, the conditon $\lambda_f(n) \neq 0$ is easily dealt with using sieve theoretical arguments, and the first expression in brackets in \eqref{eq:conditioning} is shown to be of size $\gg h$ for typical $x \in [X,2X]$. \\
On the other hand, the support of the second sum in \eqref{eq:conditioning} is shown to be sparse for typical $x$. Indeed, by exploiting the multiplicativity of $\lambda_f$ together with some Diophantine information about coefficients of cusp forms, we show that any $n$ for which $0 < |\lambda_f(n) | < X^{-\delta}$ has a prime power divisor $p^\nu > (\log X)^{c_1}$ for which $|\lambda_f(p^\nu)| < (\log X)^{-c_2}$, for $c_1,c_2 > 0$ constants depending at most on $\delta$ and $f$. Using a recent version of the Sato-Tate theorem with a quantitative error term due to Thorner \cite{Tho}, we show that the set of such multiples $n$ is a sparse set, and therefore typical length $h$ short intervals have few such multiples. This is essentially enough to conclude that the second sum in brackets in \eqref{eq:conditioning} is $o(h)$ for most $x \in [X,2X]$, giving rise to the conflicting bounds
$$
hX^{-\delta} \ll_f \sum_{x < n \leq x+h} \lambda_f(n)r(n-a) \ll_f X^{1/2+\e},
$$
(taking $\delta=\e$ small) whenever $h \gg_f X^{1/2+3\e}$. 
\subsection{Structure of the Paper}
The paper is structured as follows. In Section \ref{sec:background} we give some background results in the theory of norm forms and that of cusp forms, and we also summarize some results from multiplicative number theory of relevance in the rest of the paper. In Section \ref{sec:MainArg} we prove Theorem \ref{thm:CFNormForms}, and in Section \ref{sec:shiftedSoTS} we prove Theorem \ref{thm:consecSoTS}.

\subsection{Acknowledgements}
Most of this work was completed during a research visit at Queen's University. The author would like to thank Queen's for their support. The author is also indebted to Oleksiy Klurman, Youness Lamzouri, Maksym Radziwi\l\l{} and Asif Zaman for their advice and encouragement. 

\section{Background Results}\label{sec:background}
%For this purpose, we will invoke the following Chebotarev-Sato-Tate result, essentially due to Murty and Murty.
\subsection{Background on norm forms} \label{subsec:Bkgd}
Fix $K/\mb{Q}$ a number field, and let $\mc{N}_K$ be the sequence of norm forms of $K$. In general, the indicator function $g_K(n) := 1_{\mc{N}_K}(n)$ of $\mc{N}_K$ is not a multiplicative function, and thus an analysis of sign changes of $\lambda_f(n)g_K(n)$ purely on the basis of multiplicative techniques seems a priori difficult. However, Odoni \cite{Odo} showed that $g_K$ can be written as a linear combination of certain multiplicative functions that have fairly predictable values, and this will be sufficient for the proof of Theorem \ref{thm:CFNormForms}. The material required to these ends is drawn essentially from \cite[Sec. 13]{MRII}, which leverages Odoni's ideas. We describe the salient points in brief detail here. \\
Let $\bar{K}$ denote the normal closure of $K$, and let $\mc{H}(\bar{K})$ denote the narrow class field of $\bar{K}$. Thus, $G_K := \text{Gal}(\mc{H}(\bar{K})/\bar{K})$ is canonically isomorphic to the narrow class group $H(\bar{K})$ of $\bar{K}$, via the map $\mc{C} \in H(\bar{K}) \mapsto \sg_{\mc{C}}$, the Frobenius conjugacy class of the ideal class $\mc{C}$. \\
Consider first the case of prime norm forms.  Since the set of rational primes that ramify in $K$ is finite (they all must divide the discriminant $\text{disc}(K/\mb{Q})$) we will be able to ignore them in the sequel, and therefore focus mainly on the unramified primes. Any unramified rational prime $p$ factors as
$$
p\mc{O}_K = \mf{p}_1\cdots \mf{p}_r,
$$
where, as $N_K(p) = p^{[K:\mb{Q}]}$, we have $N_K \mf{p}_l = p^{m}$ for some $1\leq m \leq [K:\mb{Q}]$. Each of the prime ideals $\mf{p}_l$ lying above $p$ belongs to some class in the narrow class group $H(\bar{K})$. Let $\{C_1,\ldots,C_{h(K)}\}$ be an enumeration of these classes, $h(K) = |H(\bar{K})|$ denoting the narrow class number of $\bar{K}$. For each $1 \leq i \leq [K:\mb{Q}]$, $1 \leq j \leq h(K)$ and rational prime $p$ let us write
$$
b_{i,j}(p) := |\{\mf{p}|p : N_K(\mf{p}) = p^i, \mf{p} \in C_j \}|, \quad B(p) := \{b_{i,j}(p)\}_{\ss{1 \leq i \leq [K:\mb{Q}] \\ 1 \leq j \leq h(K)}}.
$$
Following Odoni, the matrix $B(p)$ is called the \emph{pattern of} $p$. Let $\mc{B} = \{B(p)\}_p$ be the (finite) collection of all pattern matrices that occur. \\
It will be profitable to have access to asymptotic formulae for the number of rational primes with a given pattern, which can be achieved by an application of the Chebotarev density theorem. Indeed, Odoni \cite[Thm. 4.1]{Odo} showed that for a given $B \in \mc{B}$ there is a collection $\mc{C}_B$ of conjugacy classes  of $G_K$ such that
$$
p \text{ is unramified and } B(p) = B \text{ if, and only if, the Frobenius class } \sg_p \in \mc{C}_B.
$$
By the Chebotarev density theorem (see \cite{ThoZam} for an unconditional result that is state-of-the-art), it thus follows that for any $B \in \mc{B}$ there is a positive constant $c(B) := \frac{|\mc{C}_B|}{h(K)} > 0$ such that as $X \ra \infty$ (keeping $K$ fixed),
\begin{equation}\label{eq:CBDforPatterns}
|\{p \leq X : \text{ unramified, } B(p) = B\}| = c(B) \int_2^X \frac{dt}{\log t} + O_K\left(X e^{-c_K\sqrt{\log X}}\right),
\end{equation}
for some constant $c_K > 0$ depending only on $K$. \\
We highlight the following consequence of this. Define
$$
\widehat{\mc{N}}_K := \{n \in \mb{N} : \exists \mf{a} \subset \mc{O}_K \text{ with } N_K(\mf{a}) = n\}.
$$
Since $N_K(\alpha) = N_K(\alpha \mc{O}_K)$ for all $\alpha \in \mc{O}_K$ we have $\mc{N}_K \subseteq \widehat{\mc{N}}_K$, though in general these sets differ. By Dedekind's theorem on factorization of integral ideals, $\widehat{\mc{N}}_K$ is a multiplicative set, i.e., $m,n \in \widehat{\mc{N}}_K$ iff $mn \in \widehat{\mc{N}}_K$.\\
As discussed by Odoni \cite[p. 71]{Odo}, there is a collection $\mf{C}$ of ideal classes such that an unramified prime $p \in \mc{N}_K$ if, and only if, $R(p) \cap \mf{C} \neq \emptyset$, where $R(p)$ is the collection of integral ideals $\mf{a} \subseteq \mc{O}_K$ such that $N_K(\mf{a}) = p$. By necessity, $\mf{a} = \mf{p}$ must then be a prime ideal, lying in some ideal class $C_j$ containing $\mf{p}$, and since $N_K(\mf{p}) \in \mf{p}$ we must have $\mf{p}|p$. Thus, there is $1 \leq j \leq h(K)$ such that $\mc{C}_j \in \mf{C}$ and $b_{1,j}(p) > 0$. In other terms, among unramified primes,
$$
p \in \mc{N}_K \text{ if, and only if, } \sum_{\ss{1 \leq j \leq h(K) \\ \mc{C}_j \in \mf{C}}} b_{1,j}(p) > 0.
$$
Define $\mc{U} \subseteq \mc{V} \subseteq \mc{B}$ the subcollections of patterns $B = \{b_{i,j}\}$ such that 
\begin{align*}
B \in \mc{U} &\text{ if and only if } \sum_{j : \mc{C}_j \in \mf{C}} b_{1,j} > 0, \\
B \in \mc{V} &\text{ if and only if } \sum_j b_{1,j} > 0.
\end{align*}
Combined with the Chebotarev density theorem, these remarks imply the following.
\begin{lem}\label{lem: countNK}
Let $\beta_K := \sum_{B \in \mc{U}} c(B)$ and $\tau_K := \sum_{B \in \mc{V}} c(B)$. Then
\begin{align}\label{eq:asympNF}
&|\{p \leq X : p \in \mc{N}_K\}| = \beta_K \int_2^X \frac{dt}{\log t} + O_K\left(Xe^{-c_K\sqrt{\log X}}\right), \\
&|\{p \leq X : p \in \widehat{\mc{N}}_K \}| = \tau_K \int_2^X \frac{dt}{\log t} + O_K\left(Xe^{-c_K\sqrt{\log X}}\right).
\end{align}
\end{lem}
%In what follows, we will define $\tau_K$ to be the constant in the main term above. \\
With the above preliminaries in hand, we can now proceed towards Odoni's decomposition theorem, as presented in \cite[Sec. 13]{MRII}. In the sequel, define the multiplicative function $\Delta_K$ at prime powers $p^{\nu}$ via
$$
\Delta_K(p^{\nu}) := \begin{cases} 1 \text{ if $p \in \widehat{\mc{N}}_K$} \\ 0 \text{ otherwise.} \end{cases}
$$
Since $\widehat{\mc{N}}_K \supseteq \mc{N}_K$ we have $\Delta_K(n) \geq g_K(n)$ for all integers $n$.
\begin{prop}[Matom\"{a}ki-Radziwi\l\l{}, Lemma 13.3 of \cite{MRII}; Odoni \cite{Odo}] \label{prop:MROdo}
There are positive real constants $\alpha = \alpha(K)$, $\rho = \rho(K)$, non-negative integers $M = M(K)$, $R = R(K)$ with $R > M$, an integer $D = D(K) \geq 1$ and complex numbers $c_i = c_i(K) \in \mb{C}$ for $0 \leq i \leq R$ such that for all $n \in \mb{N}$,
$$
g_K(n) = \sum_{0 \leq l \leq M} c_i f_i(n) + \sum_{M+1 \leq l \leq R} c_if_i(n),
$$
where the functions $f_i : \mb{N} \ra \mb{C}$ are multiplicative with $f_i(p)^D = \Delta_K(p)$ for all $p$ and $0 \leq i \leq R$, and satisfy the following properties for sufficiently large $X$: 
\begin{enumerate}[(i)]
\item For each $0 \leq l \leq R$ and $2 \leq w \leq z \leq X$,
$$
\sum_{w \leq p \leq z} \frac{|f_i(p)|}{p} = \tau_K \sum_{w \leq p \leq z} \frac{1}{p} + O_K\left(\frac{1}{\log w}\right);
$$
\item If $p \nmid \text{disc}(K/\mb{Q})$ then $f_l(p) = \Delta_K(p)$ for all $0 \leq l \leq M$; 
\item $\sum_{0 \leq l \leq M} c_l > 0$; and
\item For each $M+1 \leq l \leq R$ we have
$$
\min_{|t| \leq 2X} \sum_{p \leq X} \frac{\Delta_K(p) - \text{Re}(f_l(p)p^{-it})}{p} \geq \rho \log\log X.
$$
\end{enumerate}
\end{prop}
\noindent The utility of (iv) will become evident in Section \ref{subsec:PretNT}, where we will appeal to some notions from pretentious number theory.
\begin{proof}
Claims (ii),(iii) and (iv) are explicitly stated in \cite[Lem 13.3]{MRII}. Claim (i) follows immediately from the claim $f_l(p)^D = \Delta_K(p)$,
%, since then as $\Delta_K(p) \in \{0,1\}$ we also have $|f_l(p)| \in \{0,1\}$ and in particular 
which implies that $|f_l(p)| = \Delta_K(p) \in \{0,1\}$ for all $l$ and $p$. Thus, (i) follows from Lemma \ref{lem: countNK} and partial summation. \\
Otherwise, the only statement mentioned here that does not explicitly appear in \cite[Lem. 13.3]{MRII} is the existence of $D$. However, the functions $f_i$ are constructed in such a way that for each $B \in \mc{B}$ there is a root of unity $\zeta_{i,B} \in \mu_{d_B}$ of some order $d_B \geq 1$ such that whenever $B(p) = B$ we have $f_i(p) = \zeta_{i,B}$ \cite[p. 77]{MRII}. Since $\mc{B}$ is a finite set, this constitutes a finite set of roots of unity, and thus taking $D := \text{lcm}\{d_B : B \in \mc{B}\}$, we obtain the required number $D$. This completes the proof of the proposition.
\end{proof}

\subsection{Results from the theory of modular forms}
We will need to use the following slight generalization (to forms of weight $k > 2$) of a hybrid of the Sato-Tate and Chebotarev theorems, due to\footnote{In fact, Murty and Murty prove a slightly more general statement applied to weight $k = 2$ forms (with level $\geq 1$); the statement given here will suffice for our purposes.} M.R. Murty and V.K. Murty.
\begin{lem}[Chebotarev-Sato-Tate for modular forms of weight $k \geq 2$] \label{lem:CST}
Let $M/K$ be an Abelian Galois extension and let $G = \text{Gal}(M/K)$. Let $f$ be a Hecke eigencusp form of weight $k$ and level 1, without CM. If $\mc{C}$ is a conjugacy class of $G$ then as $X \ra \infty$ we have
$$
|\{p \leq X : \sg_p \in \mc{C}, \theta_p \in [\alpha,\beta]\}| = \left(\frac{|\mc{C}|}{|G|} \cdot \frac{2}{\pi} \int_\alpha^\beta \sin^2 \theta d\theta + o(1)\right)\pi(X)
$$
where $\theta_p$ is defined implicitly via $\lambda_f(p) = 2\cos \theta_p$, and $\sg_p$ is the Artin symbol attached to $p$.
\end{lem}
\begin{proof}
%[Proof of Lemma \ref{lem:CST}]
The argument is based on \cite[Thm. 1]{MurMur}, but is simplified since we work with Abelian $G$ (given that in our application, $G = G_K$, the narrow Hilbert class field) and also in light of the recent breakthroughs on automorphy lifting of $\text{Sym}^m\rho_f$, due to Newton and Thorne \cite{NeTh}. \\
Let $\rho_f$ be a Galois representation associated to $f$. Since $G$ is Abelian, the irreducible representations of $G$ are all 1-dimensional, thus characters. Let $\hat{G}$ be the set of characters of $G$. By Tauberian theorems and the orthogonality relations on $\hat{G}$, it suffices to show that for any $\chi \in \hat{G}$ and $m \geq 1$ we have that $L(s,\text{Sym}^m \rho_f \otimes \chi)$ is analytic and non-vanishing on $\text{Re}(s) \geq 1$. By the general theory of automorphic $L$-functions, it thus suffices to show that $L(s,\text{Sym}^m\rho_f \otimes \chi)$ is automorphic. But by the Artin reciprocity law, $\chi$ corresponds to a Hecke character $\psi$, and since $\text{Sym}^m\rho_f$ is automorphic over $\mb{Q}$ by \cite{NeTh}, so is $\text{Sym}^m \rho_f \otimes \psi$, and the claim follows. 
%It thus follows from the general theory of automorphic $L$-functions that $L(s,\text{Sym}^m\rho_f \otimes \psi)$ is analytic and non-vanishing whenever $\text{Re}(s) \geq 1$, as required.
\end{proof}

In order to keep track of sign changes we will require some control over the set of vanishing of $\lambda_f$. This will be aided by the following result of Serre (for the best result in this direction, however, see \cite{ThoZamLT})	.
\begin{lem}[Serre, Thm. 15 of \cite{Ser}] \label{lem:Ser}
Let $f$ be a Hecke eigencusp form of weight $k$ and level 1 without CM. Then for any $\e > 0$, if $X \geq X_0(\e)$ we have
$$
|\{p \leq X : \lambda_f(p) = 0\}| \ll_{\e,f} \frac{X}{(\log X)^{5/4-\e}}.
$$
\end{lem}
Let us state the following very simple consequence of Serre's result, which will arise repeatedly later on. Define
$$
B_{K,f} := \{p : \lambda_f(p) = 0\} \cup \{p : p|\text{disc}(K/\mb{Q})\}.
$$
For $n \in \mb{N}$ define the multiplicative function $\iota_{K,f}(n) := \mu^2(n) 1_{p|n \Rightarrow p \notin B_{K,f}}$ (the fact that it is supported on squarefree integers will be used later). The following lemma is a trivial consequence of Lemma \ref{lem:Ser} and partial summation.

\begin{lem}\label{lem:iotaSupp}
For any $2 \leq w \leq z \leq X$ and $\e > 0$ we have
$$
\sum_{w \leq p \leq z} \frac{1-\iota_{K,f}(p)}{p} \ll_{\e,K} \frac{1}{(\log w)^{1/4-\e}}.
$$
\end{lem}

\subsection{Results about Pretentious Number Theory}\label{subsec:PretNT}
In the sequel, our arguments dealing with multiplicative functions will, implicitly and explicitly, use notions from pretentious number theory. A key role in that theory is played by the pretentious distance functions. Write $\mb{U} := \{z \in \mb{C} : |z| \leq 1\}$. Given $y \geq 2$ and sequences $a := \{a(p)\}_p, b := \{b(p)\}_p \subset \mb{U}$ we define the pretentious distance between $a$ and $b$ up to $y$ to be
$$
\mb{D}(a,b;y) := \left(\sum_{p \leq y} \frac{1-\text{Re}(a(p)\bar{b}(p))}{p}\right)^{1/2}.
$$
This distance satisfies the \emph{pretentious triangle inequality}, i.e., if $c = \{c(p)\}_p$ is any such third sequence then
$$
\mb{D}(a,c;y) \leq \mb{D}(a,b;y) + \mb{D}(b,c;y)
$$
(see e.g. \cite[Lem. 3.1]{GSPret}). Since $\eta(u,v) := (1-\text{Re}(u\bar{w}))^{1/2}$, well-defined for $u,v \in \mb{U}$, satisfies $\eta(u,v) = \eta(u\bar{v},1) = \eta(1,\bar{u}v)$, one can consequently show the useful inequality
$$
\mb{D}(a_1a_2,b_1b_2;y) \leq \mb{D}(a_1,b_1;y) + \mb{D}(a_2,b_2;y),
$$
using the notation $(ab)(p) := a(p)b(p)$ for all $p$ for prime-indexed sequences $a$ and $b$. Among other things, iterating this leads to the result that 
\begin{equation}\label{eq:pretwithDpowers}
\mb{D}(a^d,b^d;y) \leq d\mb{D}(a,b;y) \text{ for any } d \in \mb{N}.
\end{equation}
Perhaps less well-known is the following general, weighted variant of the pretentious triangle inequality, also due to Granville and Soundararajan \cite{GSPret}.
\begin{lem} \label{lem:wtdTriEq}
Let $\{r(p)\}_p\subset [0,\infty)$, and let $a = \{a(p)\}_p, b = \{b(p)\}_p$ and $c = \{c(p)\}_p$ be as above. Then for any $y \geq 2$,
$$
\left(\sum_{p \leq y} r(p)(1-\text{Re}(a(p)\bar{c}(p)))\right)^{1/2} \leq \left(\sum_{p \leq y} r(p)(1-\text{Re}(a(p)\bar{b}(p)))\right)^{1/2} + \left(\sum_{p \leq y} r(p)(1-\text{Re}(b(p)\bar{c}(p)))\right)^{1/2}.
$$
%In particular, if $S$ is any sequence of primes 
\end{lem}
\begin{proof}
This is alluded to below the proof of Lem 3.1 in \cite{GSPret}. Since the proof is not given, we give the details here, elaborating slightly on the arguments in \cite{GSPret} for the sake of clarity. \\
Observe first of all that if $z \in \mb{U}$ then $|z| \leq 1$ and thus
$$
\text{Im}(z)^2 \leq 1- \text{Re}(z)^2 = (1+\text{Re}(z))(1-\text{Re}(z)) \leq 2(1-\text{Re}(z)).
$$ 
It follows, therefore, that for any $z_1,z_2 \in \mb{U}$, $2\eta(1,z_1)\eta(1,z_2) \geq |\text{Im}(z_1)||\text{Im}(z_2)|$. Next, note the property
\begin{align*}
&0 \leq (1-\text{Re}(z_1))(1-\text{Re}(z_2)) = 1 - \text{Re}(z_1) - \text{Re}(z_2) + \text{Re}(z_1)\text{Re}(z_2), 
\end{align*}
whence we may deduce that
\begin{align*}
&\text{Re}(z_1\bar{z_2}) = \text{Re}(z_1)\text{Re}(z_2) + \text{Im}(z_1)\text{Im}(z_2) \geq \text{Re}(z_1)+\text{Re}(z_2) - 1 + \text{Im}(z_1)\text{Im}(z_2).
\end{align*}
We therefore obtain that
\begin{align*}
\eta(z_1,z_2)^2 &= 1-\text{Re}(z_1\bar{z_2}) \leq 2 - \text{Re}(z_1)-\text{Re}(z_2)-\text{Im}(z_1)\text{Im}(z_2) \\
&\leq \eta(1,z_1)^2 + \eta(1,z_2)^2 + |\text{Im}(z_1)||\text{Im}(z_2)| \leq \eta(1,z_1)^2 + \eta(1,z_2)^2 + 2\eta(1,z_1)\eta(1,z_2) \\
&= (\eta(1,z_1)+\eta(1,z_2))^2.
\end{align*}
Specializing to the case $z_1= u\bar{v}$ and $z_2 = v\bar{w}$, we obtain the preliminary inequality
\begin{equation}\label{eq:prelimTriEq}
\eta(u,w) \leq \eta(u,v) + \eta(v,w),
\end{equation}
for any $u,v,w\in \mb{U}$. \\
Next, note that by Cauchy-Schwarz we have
$$
\sum_{p \leq y} r(p) \eta(a(p),b(p))\eta(b(p),c(p)) \leq \left(\sum_{p \leq y} r(p) \eta(a(p),b(p))^2\right)^{1/2} \left(\sum_{p \leq y} r(p) \eta(b(p),c(p))^2\right)^{1/2},
$$
and from this and \eqref{eq:prelimTriEq} we obtain
\begin{align*}
\left(\left(\sum_{p \leq y} r(p) \eta(a(p),b(p))^2\right)^{1/2} + \left(\sum_{p \leq y} r(p)\eta(b(p),c(p))^2\right)^{1/2}\right)^2 &\geq \sum_{p \leq y} r(p)(\eta(a(p),b(p)) + \eta(b(p),c(p)))^2 \\
&\geq \sum_{p \leq y} r(p)\eta(a(p),c(p))^2.
\end{align*}
This implies the claim.
\end{proof}
We specialize the above lemma as follows. Given $y \geq 2$, a number field $K/\mb{Q}$ and multiplicative functions $\phi_1,\phi_2: \mb{N} \ra \mb{C}$, let us write
$$
\mb{D}_K(\phi_1,\phi_2;y) := \left(\sum_{\ss{p \leq y \\ p \in \widehat{\mc{N}}_K}} \frac{1-\text{Re}(\phi_1(p)\bar{\phi_2}(p))}{p}\right)^{1/2},
$$
i.e., we set $r(p) := \Delta_K(p)/p$ in the notation of the previous lemma. Trivially, $\mb{D}_{\mb{Q}}(\cdot,\cdot;y) = \mb{D}(\cdot,\cdot;y)$ for all $y\geq 2$. Note in particular that item (iv) of Proposition \ref{prop:MROdo} is precisely the statement that $\min_{|t| \leq 2X} \mb{D}_K(f_l,n^{it};X)^2 \geq \rho \log\log X$ for all $M+1 \leq l \leq R$. The above lemma allows us to conclude, analogously to \eqref{eq:pretwithDpowers}, that for any $d \geq 1$,
\begin{equation} \label{eq:pretwithDpow}
\mb{D}_K(f^d,g^d;X) \leq d\mb{D}_K(f,g;X),
\end{equation}
for $f,g$ multiplicative functions taking values in $\mb{U}$. \\
In the sequel, we take the convention that $\text{sign}(0) := 0$, and for $n \in \mb{N}$ we define the multiplicative function $\sg_f(n) := \text{sign}(\lambda_f(n))$, which takes values in $\{-1,0,1\}$. By Lemma \ref{lem:Ser} we know that $\sg_f(p) \neq 0$ for all but a zero density set of primes, and in fact by the Sato-Tate theorem we have
$$
|\{p \leq X: \sg_f(p) = +1 \}| \sim |\{p \leq X : \sg_f(p) = -1\}| \sim \frac{1}{2}\pi(X).
$$
At several junctures of our argument we will require precise information about, in particular quantitative lower bounds for, $\mb{D}_K(\sg_f,n^{it};X)$, where $|t| \leq 2X$, among other such distances. In this direction, we prove the following. 
%Write $\mu_d$ to denote the set of $d$th roots of unity, for $d \geq 1$.
\begin{lem} \label{lem:pretDistBd}
Let $0 \leq l \leq R$, and let $f_l$ be one of the functions mentioned in Proposition \ref{prop:MROdo}. Then there is a $\sg = \sg(K) > 0$ such that whenever $|t| \leq 2X$ we have
$$
\mb{D}_K(\sg_f f_l, n^{it};X)^2 \geq \sg \min\{\log\log X, \log(1+|t|\log X)\} - O_{K,f}(1).
$$
Moreover, as $X \ra \infty$ we have
$$
\min_{|t| \leq 2X} \mb{D}_K(\sg_f f_l, n^{it};X)^2 \ra \infty.
$$
\end{lem}
\begin{proof}
%Our proof splits into two ranges, namely the range $1 \leq |t| \leq 2X$ and the range $|t| \leq 1$. \\
%\underline{Case 1:} $1 \leq |t| \leq 2X$ \\
We begin with the first claim.
Recall from Proposition \ref{prop:MROdo} that there is a $D = D(K) \in \mb{N}$ such that $f_l^D = \Delta_K$, and thus $(f_l(p) \sg_f(p))^{2D} = 1$ for all $p$ such that $\Delta_K(p) \lambda_f(p) \neq 0$. It follows from Lemma \ref{lem:iotaSupp} and \eqref{eq:pretwithDpow} that
\begin{align*}
(2D) \mb{D}_K(\sg_f f_l, n^{it};X) \geq \mb{D}_K((\sg_f f_l)^{2D}, n^{2iDt};X) 
%&\geq (2D) \left(\sum_{\ss{p \leq X \\ \Delta_K(p) \lambda_f(p) \neq 0}} \frac{1-\text{Re}(\sg_f(p)f_l(p)p^{-it})}{p}\right)^{1/2} \\
&= \left(\sum_{\ss{p \leq X \\ \Delta_K(p) \lambda_f(p) \neq 0}} \frac{1-\text{Re}(p^{-2Dit})}{p}\right)^{1/2} \\
&\geq \mb{D}_K(1,n^{2iDt};X) - O_{K,f}(1).
\end{align*}
%In particular, we have $\mb{D}_K(\sg_f f_l, n^{it};X)^2 \geq (4D^2)^{-1} \mb{D}_K(1,n^{2iDt};X)^2 - O_{K,f}(1)$. \\
If $D|t| \leq 10/\log X$ then the claim is trivial, so we may assume otherwise that $|t| > 10/\log X$. 
Now, for $\eta \in (0,1/3)$ set $N_{X,\eta} := \exp((\log X)^{2/3+\eta})$ and let $Y := \max\{N_{X,\eta}, e^{1/|Dt|}\} \leq X$. By \cite[Lem. 13.3(ii)]{MRII} we have
$$
\mb{D}_K(1,n^{2iDt};X)^2 \geq \sum_{B \in \mc{V}} \text{Re}\left( \sum_{\ss{Y < p \leq X \\ B(p) = B}} \frac{1-p^{2iDt}}{p}\right) + O_K(1) = \tau_K \sum_{Y < p \leq X} \frac{1-\text{Re}(p^{2iDt})}{p} + O_K(1).
$$
We now consider two cases. If $1 \leq D|t| \leq 2X$ then it follows from standard arguments employing the Vinogradov-Korobov zero-free region for the Riemann zeta function that the latter prime sum is
$$
\tau_K(\tfrac 13 - \eta) \log\log X - O_K(1).
$$
Thus, for $1 \leq |t| \leq 2X$ we have
$$
\mb{D}_K(\sg_f f_l,n^{it};X)^2 \geq \frac{(1-3\eta)\tau_K}{12D^2} \log\log X - O_K(1).
$$
On the other hand, if $10/\log X < D|t| \leq 1$ then by partial summation and the prime number theorem we have
\begin{align*}
\sum_{Y < p \leq X} \frac{1-\text{Re}(p^{2iDt})}{p} &= \frac{1}{2\pi} \left(\int_0^{2\pi} (1-\cos \alpha)d\alpha\right) \cdot \log\left(\frac{\log(2D|t| \log X)}{\max\{1, D|t|\log N_{X,\eta}\}}\right) - O_K(1)\\
&\geq \left(\frac{1}{3}-\eta\right) \min\{\log\log X, \log(1+|t| \log X)\} - O_K(1).
\end{align*}
This completes the proof of the first claim with any $0 < \sg < \tau_K/(12D^2)$ (choosing $\eta$ appropriately small). \\
For the second, let $Z = Z(X) \geq 10$ be a parameter to be chosen later and observe that if $|t| \geq Z/\log X$ then by the first claim we obtain
$$
\min_{Z/\log X \leq |t| \leq 2X} \mb{D}_K(\sg_f f_l, n^{it};X)^2 \geq \sg \log Z.
$$
Next, suppose $|t| \leq Z/\log X$. 
%\underline{Case 2:} $|t| \leq 1$ \\
As mentioned in the proof of Proposition \ref{prop:MROdo}, Odoni showed that $f_l$ is constant on the set of unramified primes $p \in \widehat{\mc{N}}_K$ with a fixed pattern $B(p) = B$, taking a root of unity depending on $B$ as its value. Thus, let $B \in \mc{V}$. We know that $f_j(p) = \zeta_B$ for some root of unity (possibly equal to 1) for all $p$ with $B(p) = B$. 
As discussed in Section \ref{subsec:Bkgd}, there is a collection of conjugacy classes $\mc{C}_B$ of $G_K$ for which the Frobenius $\sg_p$ of $p$ belongs to some class $C \in \mc{C}_B$. 
Applying \eqref{eq:CBDforPatterns} and partial summation, whenever $|t| \leq Z/\log X$ we obtain
\begin{align*}
\mb{D}_K(\sg_f f_j, n^{it};X)^2 &\geq  \sum_{\ss{e^{\sqrt{\log X}} < p \leq X \\ B(p) = B}} \frac{1-\sg_f(p)\text{Re}(\zeta_B p^{-it})}{p} \\
&= \frac{c(B)}{2} \log\log X + O_K(1) - \text{Re}\left(\zeta_B \sum_{C \in \mc{C}_B} \sum_{\ss{e^{\sqrt{\log X}} < p \leq X \\ \sg_p \in C}} \frac{\sg_f(p)p^{-it}}{p}\right),
\end{align*}
where the contribution from ramified primes is absorbed in the $O_K(1)$ term.
Applying partial summation, the estimate $|t| \leq Z/\log X$ and Lemma \ref{lem:CST}, we obtain that there is a function $\e(t) \ra 0$ as $t \ra \infty$ such that for each $C \in \mc{C}_B$,
\begin{align*}
&\left|\sum_{\ss{e^{\sqrt{\log X}} < p \leq X \\ \sg_p \in C}} \frac{\sg_f(p)p^{-it}}{p}\right| \ll (1+|t|) \int_{e^{\sqrt{\log X}}}^X \left|\sum_{\ss{p \leq Y \\ \sg_p \in C}} \sg_f(p)\right| \frac{du}{u^2} + \frac{1}{\sqrt{\log X}} \\
&\leq (Z \log\log X) \max_{e^{\sqrt{\log X}} \leq Y \leq X} \frac{1}{\pi(Y)} \left||\{p \leq Y : \sg_p \in C, \sg_f(p) \in [0,1]\}| - |\{p \leq Y : \sg_p \in C , \sg_f(p) \in [-1,0]\}| \right| + \frac{1}{\sqrt{\log X}}\\
&\leq \left(\min_{e^{\sqrt{\log X}} \leq Y\leq X} \e(Y)\right) Z\log\log X.
\end{align*}
If we therefore define $Z=Z(X)$ via
$$
Z(X)^2 := \left(\min_{e^{\sqrt{\log X}} \leq Y \leq X} \e(Y)\right)^{-1},
$$
we obtain that, uniformly over $|t| \leq 2X$,
$$
\sum_{p \leq X} \frac{\Delta_K(p) - \sg_f(p)\text{Re}(f_j(p)p^{-it})}{p} \geq \min\{\sg \log Z, \left(\frac{1}{2}\max_{B \in \mc{V}} c(B) - O(\tfrac{1}{Z})\right) \log\log X\} - O_{K,f}(1),
$$
and as $Z(X) \ra \infty$ with $X$, we obtain the second claim.
\end{proof}

\subsection{Background on multiplicative functions}
In the proof of our main theorems, we will need to estimate, both asymptotically and/or with non-trivial upper bounds, the averages of sparsely-supported multiplicative functions. The following two lemmas will be key in this endeavour.
The first of these results, due to Wirsing, allows us to asymptotically estimate partial sums of non-negative multiplicative functions that are slowly growing and suitably regular on the primes; it applies immediately to any such function that is everywhere bounded by 1.
\begin{lem}[Wirsing \cite{Wir}] \label{lem:Wir}
Let $\phi: \mb{N} \ra [0,\infty)$ be a multiplicative function for which there are constants $A,B > 0$ such that
$$
\sup_p \phi(p) \leq B, \quad \quad \sum_{p^{\nu}, \nu \geq 2} \frac{\phi(p^{\nu}) \log p^{\nu}}{p^{\nu}} \leq A.
$$
Assume furthermore that there is $\tau > 0$, such that 
$$
\sum_{p \leq X} \frac{\phi(p) \log p}{p} = (\tau+o(1)) \log X.
$$
Then, as $X \ra \infty$,
$$
\frac{1}{X}\sum_{n \leq X} \phi(n) = \left(\frac{e^{-\gamma \tau}}{\Gamma(\tau)} + o(1)\right) \prod_{p \leq X} \left(1-\frac{1}{p}\right) \sum_{k \geq 0} \frac{\phi(p^k)}{p^k}.
$$
\end{lem}

We would also like a result complementing Wirsing's theorem that allows us to say that if $g$ is a multiplicative function such that $|g(n)|$ satisfies the hypotheses of Lemma \ref{lem:Wir} then, provided $g$ oscillates sufficiently, we have $\sum_{n \leq X} g(n) = o(\sum_{n \leq X} |g(n)|)$. The following strong result of this kind is due to Tenenbaum.
\begin{lem}[Tenenbaum, Cor. 2.1 of \cite{TenVM}] \label{lem:Ten}
Let $T \geq 1$. Let $g : \mb{N} \ra \mb{C}$ be a multiplicative function such that
$$
\sup_p |g(p)| \leq A, \quad \sum_{p^{\nu}, \nu \geq 2} \frac{|g(p^{\nu})| (\log p^{\nu})^2}{p^{\nu}} \leq B.
$$
Assume furthermore that there is a constant $\beta > 0$ such that for any $2 \leq y \leq x$ we have
$$
\sum_{y \leq p \leq x} \frac{|g(p)|}{p} \geq \beta \log\left(\frac{\log x}{\log y}\right) + O(1).
$$
Then as $X \ra \infty$,
$$
\left|\frac{1}{X}\sum_{n \leq X} g(n)\right| \ll_{A,B,\beta} \left(\frac{1}{X}\sum_{n \leq X} |g(n)|\right) \cdot \left(\frac{1+m_g(X;T)}{e^{m_g(x;T)}} + \frac{1}{\sqrt{T}} + \frac{1}{\log x}\right),
$$
where we have denoted
$$
m_g(x;T) := \min_{|t| \leq T} \sum_{p \leq X} \frac{|g(p)|-\text{Re}(g(p)p^{-it})}{p}.
$$
\end{lem}

%The final piece of the puzzle is 
Finally, we will need a means of comparing short and long interval averages of multiplicative functions. To this end we will apply a recent result of the author \cite{MRDB}, which in this context is a slight refinement (in certain aspects) of \cite[Thm. 1.9]{MRII}. To state it we will need a bit of notation. For parameters  $0 < \sg \leq A \leq 1$, $\gamma > 0$ and $X$ large we define the set\footnote{In the notation of \cite{MRDB}, the collection of such multiplicative functions is denoted as $\mc{M}(X;A,1,1;\gamma,\sg)$t.} $\mc{M}(X;A,\gamma,\sg)$ to be the collection of all multiplicative functions $g: \mb{N} \ra \mb{C}$ such that 
\begin{enumerate}
\item $|g(n)| \leq 1$ for all $n \leq X$; 
\item for all $z_0 \leq z \leq w \leq X$ we have\footnote{The need to appeal to \cite{MRDB} rather than \cite{MRII} lies here, wherein we have the flexibility of choosing $\gamma < 1$; in order to apply the main Theorem 1.9 in \cite{MRII} one would require $\gamma = 1$ instead, which will not be available for us given the current best available unconditional estimates for $|\{p \leq X: \lambda_f(p) = 0\}|$ (even using the work of \cite{ThoZamLT} in place of Lemma \ref{lem:Ser}.}
$$
\sum_{z < p \leq w} \frac{|g(p)|}{p} \geq A \sum_{z < p \leq w} \frac{1}{p} - O\left(\frac{1}{(\log z)^{\gamma}}\right);
$$
\item if $t_0 = t_0(g;X) \in [-X,X]$ is a minimizer for the map 
$$
t \mapsto \rho(g,n^{it};X)^2 := \sum_{p \leq X} \frac{|g(p)| - \text{Re}(g(p)p^{-it})}{p},
$$
then for every $t \in [-2X,2X]$ we have 
$$
\rho(g,n^{it};X)^2 \geq \sg \min\{\log\log X, \log(1+|t-t_0|\log X)\} - O_A(1).
$$
\end{enumerate}
For $g \in \mc{M}(X;A,\gamma,\sg)$ we also define
$$
H(g;X) := \prod_{p \leq X} \left(1+\frac{(|g(p)|-1)^2}{p}\right).
$$

\begin{lem}[\cite{MRDB}, Thm. 1.7] \label{lem:MTManDB}
Let $X \geq 100$. Let $g \in \mc{M}(X;A,\gamma,\sg)$ and put $t_0 = t_0(g;X)$. 	Let $10 \leq h_0 \leq X/10H(g;X)$ and set $h:= h_0 H(g;X)$. Then there is $\kappa = \kappa(A,\sg) > 0$ such that
\begin{align*}
&\frac{1}{X}\int_X^{2X}\left|\frac{1}{h}\sum_{x \leq n \leq x+h} g(n) - \frac{1}{h}\int_{x}^{x+h} u^{it_0} du \cdot \frac{1}{X}\sum_{X \leq n \leq 2X} g(n)n^{-it_0}\right|^2 dx \\
&\ll_A \left(\left(\frac{\log\log h_0}{\log h_0}\right)^A + \frac{\log\log X}{(\log X)^{\kappa}}\right) \prod_{p \leq X} \left(1+\frac{|g(p)|-1}{p}\right)^2.
\end{align*}
When $g$ is non-negative the above estimate holds with $t_0 = 0$.
\end{lem}
\begin{rem}
The statement that $t_0 = 0$ when $g$ is a non-negative function is not explicitly stated in \cite[Thm. 1.7]{MRDB}, but the fact that $0$ is a minimizer of $t \mapsto \rho(g, n^{it};X)$ is clear: indeed, $\rho(g,n^{it};X) \geq 0$ trivially for all $t$, and moreover as $g(p) = \text{Re}(g(p)) = |g(p)|$ in this case we have
$$
\rho(g,1;X)^2 = \sum_{p \leq X} \frac{|g(p)|-\text{Re}(g(p))}{p} = 0.
$$
\end{rem}
As a consequence of this result, we obtain the following.
\begin{cor}\label{cor:cardCloseSums}
Let $X$ be large and let $10 \leq h_0 \leq (\log X)^{100}$. Let $g \in \mc{M}(X;A,\gamma,\sg)$ and set $t_0 = t_0(g;X)$. Let $\delta = \delta(X) \in (0,1)$ be a small quantity such that
$$
\delta^{3/A} \geq \frac{\log \log h_0}{\log h_0}.
$$
Then for all but $O(\delta X)$ integers $x \in [X,2X]$ we have
$$
\left|\frac{1}{h} \sum_{x \leq n \leq x+h} g(n) - \frac{1}{h}\int_x^{x+h} u^{it_0} du \cdot \frac{1}{X}\sum_{X \leq n \leq 2X} g(n)n^{-it_0}\right| \ll \delta \prod_{p \leq X}\left(1+\frac{|g(p)|-1}{p}\right).
$$
If $g$ is non-negative then the same estimate holds with $t_0 = 0$.
\end{cor}
\begin{proof}
Since $g \in \mc{M}(X;A,\gamma,\sg)$, Lemma \ref{lem:MTManDB} together with Chebyshev's integral inequality yields that the measure of $x \in [X,2X]$ for which the claimed bound does not hold is
\begin{align*}
&\leq \delta^{-2}\prod_{p \leq X} \left(1+ \frac{|g(p)|-1}{p}\right)^{-2} \int_X^{2X} \left|\frac{1}{h} \sum_{x \leq n \leq x+h} g(n) -  \frac{1}{h}\int_x^{x+h} u^{it_0} du \cdot \frac{1}{X}\sum_{X \leq n \leq 2X} g(n)n^{-it_0} \right|^2 dx \\
&\ll \delta^{-2}X\left(\frac{\log\log h_0}{\log h_0}\right)^A \leq \delta X.
\end{align*}
As in Lemma \ref{lem:MTManDB}, when $g$ is non-negative the same conclusion applies with $t_0 = 0$. Since $x\mapsto \sum_{x \leq n \leq x+h} g(n)$ is piecewise constant, the claimed cardinality bound for integer $x \in [X,2X]$ is equivalent to the measure bound for real $x \in [X,2X]$.  
\end{proof}

\section{Proof of Theorem \ref{thm:CFNormForms}} \label{sec:MainArg}
In this section we will prove Theorem \ref{thm:CFNormForms}. We briefly describe here the idea of the argument. Recall the notation 
$$
\delta_K(X) = \prod_{\ss{p \leq X \\ p \notin \widehat{\mc{N}}_K}} (1-1/p).
$$ Suppose we can show that for large enough $X$ and $h = h(X) \ll \delta_K(X)^{-1}$ there are $\geq 3X/4$ integers $x \in [X,2X]$ for which each of the following inequalities hold:
\begin{equation} \label{eq:nonzeroShortsums}
\sum_{x \leq n \leq x+h} g_K(n)(|\sg_f(n)| + \sg_f(n)) > 0 \quad \text{ and }  \sum_{x \leq n \leq x+h} g_K(n)(|\sg_f(n)| - \sg_f(n)) > 0.
\end{equation}
By the union bound, it would follow that for $\geq X/2$ integers $x \in [X,2X]$ there are integers $n_1,n_2 \in [x,x+h]$ such that $n_j \in \mc{N}_K$ and $\sg_f(n_1)\sg_f(n_2) < 0$. This implies the existence of a sign change for $\lambda_f(n)$ as $n$ ranges over norm forms $\mc{N}_K \cap [x,x+h]$ in $\geq X/2$ such sets. Choosing greedily a subset of these endpoints $x$ such that each pair of consecutive points are $>h$ from one another, we deduce that each of these sign changes is distinct, and their number is $\gg X/h \gg \delta_K(X) X$. This would thus yield Theorem \ref{thm:CFNormForms}. \\
In order to verify \eqref{eq:nonzeroShortsums} we will compare, using Corollary \ref{cor:cardCloseSums}, such short sums to corresponding long sums. We will then apply the results of the previous section in order to estimate these long sums. \\
Recall once again that
$$
\iota_{K,f}(n) := \mu^2(n)1_{p|n \Rightarrow p \notin B_{K,f}},
$$
where $B_{K,f} := \{p: \lambda_f(p) = 0\} \cup \{p : p|\text{disc}(K/\mb{Q})\}$. 

\begin{lem} \label{lem:longSumBounds}
With the notation of Proposition \ref{prop:MROdo}, let $M+1 \leq l \leq R$, let $u \in [-X,X]$ and let $g \in \{1,\sg_f\}$. 
%Let $\sg = \sg(K) \in (0,1)$, $\rho = \rho(K)$ be the constants arising in Lemma \ref{lem:pretDistBd} and Proposition \ref{prop:MROdo}, and set $\theta := \min\{\rho,\sg\}$. 
Then the following bounds hold as $X \ra \infty$:
\begin{align}
&\frac{1}{X}\sum_{X \leq n \leq 2X} \Delta_K(n) \iota_{K,f}(n) 
%= \left(\frac{e^{-\gamma \tau_K}}{\Gamma(\tau_K)} + o(1)\right) \prod_{p \leq X} \left(1-\frac{1}{p}\right) \left(1+\frac{\Delta_k(p)}{p}\right) 
\asymp_{K,f} \prod_{p \leq X} \left(1+\frac{\Delta_K(p)-1}{p}\right) \label{eq:Bd1}\\
&\left|\frac{1}{X}\sum_{n \leq X} \sg_f(n)\Delta_K(n)\iota_{K,f}(n)n^{-iu}\right| = o_{K,f}\left(\prod_{p \leq X} \left(1+\frac{\Delta_K(p)-1}{p}\right)\right) \label{eq:Bd2}\\
&\left|\frac{1}{X}\sum_{n \leq X} g(n)f_l(n)\iota_{K,f}(n) n^{-iu}\right| = o_{K,f}\left(\prod_{p \leq X} \left(1+\frac{\Delta_K(p)-1}{p}\right)\right). \label{eq:Bd3}
\end{align}
\end{lem}
\begin{proof}
Note that $0 \leq \Delta_K\iota_{K,f} \leq 1$, and combining Lemmas \ref{lem: countNK} and \ref{lem:iotaSupp} with partial summation we have
$$
\sum_{p \leq X} \frac{\Delta_K(p)\iota_{K,f}(p) \log p}{p} = \sum_{p \leq X} \frac{\Delta_K(p) \log p}{p} + O\left(\sum_{p \leq X} \frac{(1-\iota_{K,f}(p))\log p}{p}\right) = \tau_K \log X + O_{\e}((\log X)^{3/4+\e}).
$$
Hence, Lemma \ref{lem:Wir} applies, and we obtain
$$
\frac{1}{X}\sum_{X \leq n \leq 2X} \Delta_K(n) \iota_{K,f}(n) =
\left(\frac{e^{-\gamma \tau_K}}{\Gamma(\tau_K)} + o(1)\right) \prod_{p \leq X} \left(1-\frac{1}{p}\right) \left(1+\frac{\Delta_K(p) \iota_{K,f}(p)}{p}\right).
$$
By Mertens' theorem and Lemma \ref{lem:iotaSupp}, we obtain that
\begin{align*}
\prod_{p \leq X} \left(1-\frac{1}{p}\right)\left(1+\frac{\Delta_K(p)\iota_{K,f}(p)}{p}\right) \asymp \exp\left(\sum_{p \leq X} \frac{\Delta_K(p)\iota_{K,f}(p)-1}{p}\right) &= \exp\left(\sum_{p \leq X} \frac{\Delta_K(p)-1}{p} + O_{K,f}(1)\right) \\
&\asymp_{K,f} \prod_{p \leq X} \left(1+\frac{\Delta_K(p)-1}{p}\right),
\end{align*}
which in light of the previous estimate gives \eqref{eq:Bd1}. \\
Estimates \eqref{eq:Bd2} and \eqref{eq:Bd3} are consequences of Lemma \ref{lem:pretDistBd}. Indeed, we have $|\sg_f \Delta_K n^{-iu}|, |gf_l n^{-iu}| \leq \Delta_K$, with equality at primes except on the sparse set $B_{K,f}$. In the notation of Lemma \ref{lem:Ten}, taking $T = (\log X)^2$,
\begin{align*}
m_{\phi}(X;T) &= \min_{|t| \leq T} \sum_{p \leq X} \frac{\Delta_K(p) - \text{Re}(\phi(p)p^{-it})}{p} - O_{K,f}(1) \\
&\geq \min_{|t| \leq T} \mb{D}_K(\phi ,n^{i(t-u)};X)^2 - O_{K,f}(1) \ra \infty,
\end{align*}
as $X \ra \infty$ whenever $\phi = \sg_f \Delta_Kn^{-iu}$ or $\phi = g f_l n^{-iu}$, for $M+1 \leq l \leq R$ and $|u| \leq X$. Finally, if, more simply, $\phi = f_ln^{-iu}$ then by Proposition \ref{prop:MROdo} we have
$$
m_{\phi}(X;T) \geq \min_{|t| \leq 2X} \mb{D}_K(f_l,n^{it};X)^2 - O_{K,f}(1) \geq \rho \log\log X - O_{K,f}(1).
$$
Note that Lemma \ref{lem:Ten} applies since all of the functions in play are bounded by 1, and the lower bound on sums $\sum_{y \leq p \leq x} \Delta_K(p)/p$ is an immediate consequence of Lemma \ref{lem: countNK}. We thus get that each of the quantities on the LHS of \eqref{eq:Bd2} and \eqref{eq:Bd3} is bounded from above by
$$
\ll_{K,f} \left(\frac{1}{X} \sum_{n \leq X} \Delta_K(n)\right) \cdot \left(\frac{1+m_{\phi}(X;(\log X)^2)}{e^{m_{\phi}(X;(\log X)^2)}} + \frac{1}{\log X}\right) = o_{K,f}\left(\prod_{p \leq X} \left(1+\frac{\Delta_K(p)-1}{p}\right)\right),
$$
where in the last step the bracketed sum was estimated using Lemma \ref{lem:Wir}, just as in \eqref{eq:Bd1}. This completes the proof.
\end{proof}

Finally, let us note that as $\Delta_K(p) \in \{0,1\}$ we have
$$
H(\Delta_K;X) = \prod_{p \leq X} \left(1+\frac{(\Delta_K(p)-1)^2}{p}\right) = \prod_{p \leq X} \left(1+\frac{1-\Delta_K(p)}{p}\right)  = \delta_K(X)^{-1}.
$$
We will use this fact in what follows.

\begin{proof}[Proof of Theorem \ref{thm:CFNormForms}]
Let $\delta \in (0,1/4)$ be a small parameter. In light of the discussion at the beginning of this section, it will suffice to show that for all but $O(\delta X)$ choices of $x \in [X,2X]$ we obtain, for both $\eta \in \{-1,+1\}$, the positive lower bound
$$
\sum_{x \leq n \leq x+h} g_K(n)(|\sg_f(n)| + \eta \sg_f(n)) > 0.
% \quad \text{ and } \quad \sum_{x \leq n \leq x+h} g_K(n)(|\sg_f(n)| - \sg_f(n)) > 0.
$$
Since $|\sg_f(n)| +\eta \sg_f(n) \geq 0$ for $\eta \in \{-1,+1\}$, by positivity the latter sum is
$$
\geq \sum_{x \leq n \leq x + h} g_K(n) \iota_{K,f}(n)(|\sg_f(n)| + \eta \sg_f(n)) = \sum_{x \leq n \leq x+h} g_K(n)\iota_{K,f}(n)(1+\eta \sg_f(n)),
$$
since $\iota_{K,f}(n) = 0$ whenever $\sg_f(n) = 0$. 
By Proposition \ref{prop:MROdo} we may use the decomposition of $g_K$ as a linear combination of multiplicative functions to write this last sum as $\mc{L}(x) + \mc{S}(x)$, where we define
\begin{align*}
\mc{L}(x) &:= \sum_{0 \leq l \leq M} c_l \sum_{x \leq n \leq x+h} f_l(n)\iota_{K,f}(n)(1 + \eta \sg_f(n)) \\
&= \left(\sum_{0 \leq l \leq M} c_l\right) \sum_{x \leq n \leq x+h} \Delta_K(n)\iota_{K,f}(n)(1+\eta \sg_f(n)), \\
\mc{S}(x) &:= \sum_{M+1 \leq l \leq R} c_l \sum_{x \leq n \leq x+h} f_l(n)\iota_{K,f}(n)(1 + \eta \sg_f(n)),
\end{align*}
where we reexpressed $\mc{L}$ using the fact that $\iota_{K,f}(n) = 0$ unless $(n,\text{disc}(K/\mb{Q})) = 1$, in which case $f_l(n) = \Delta_K(n)$ for all $0 \leq l \leq M$. Combining Proposition \ref{prop:MROdo} with Lemmas \ref{lem:Ser} and \ref{lem:pretDistBd}, note that $\phi \in \mc{M}(X;\tau_K, 1/4-\eta, (1-3\eta)\tau_K/(12D^2)\}$ whenever $\phi = \sg_f f_l$ for any $0 \leq l \leq R$ and $\eta > 0$ is sufficiently small; when $\phi = f_l$ the same lemmas also imply that $\phi \in \mc{M}(X;\tau_K, 1/4-\eta,\rho)$.\\
Let us show first that 
$$
\mc{S}(x) \ll_{K,f} \delta h\prod_{p \leq X} \left(1+\frac{\Delta_K(p)-1}{p}\right) \text{ for all but $O_{K,f}(\delta X)$ choices of $x \in [X,2X]$.}
$$
To see this, we note first of all that for any $t_0 \in [-X,X]$ and $g \in \{1,\sg_f\}$, Lemma \ref{lem:longSumBounds} shows that for each $M+1 \leq l \leq R$,
$$
\frac{1}{X}\left|\sum_{X \leq n \leq 2X} f_l(n)g(n)\iota_{K,f}(n)n^{-it_0}\right| \ll \max_{Y \in [X,2X]} \frac{1}{Y}\left|\sum_{n \leq Y} f_l(n)g(n)n^{-it_0}\right| = o_{K,f}\left(\prod_{p \leq X}\left(1+\frac{\Delta_K(p)-1}{p}\right)\right).
$$
Moreover, choosing $X$ sufficiently large, $h_0$ sufficiently large in terms of $\delta$ and writing $h = h_0H(\Delta_K;X)$, we may combine this with Corollary \ref{cor:cardCloseSums} and the triangle inequality to deduce that for each $M+1 \leq l \leq R$, for all but $O(\delta X)$ choices of $x \in [X,2X]$ we have
\begin{align*}
&\left|\frac{1}{h}\sum_{x \leq n \leq x+h} f_l(n)g(n)\iota_{K,f}(n) \right| \\
&\ll \left|\frac{1}{h}\sum_{x \leq n \leq x+h} f_l(n)g(n)\iota_{K,f}(n) - \frac{1}{h}\int_x^{x+h} u^{it_0} du \cdot \frac{1}{X}\sum_{X \leq n \leq 2X} f_l(n)g(n)\iota_{K,f}(n) n^{-it_0}\right| \\
&+ \frac{1}{X}\left|\sum_{X \leq n \leq 2X} f_l(n)g(n)\iota_{K,f}(n)n^{-it_0}\right| \\
&\ll_{K,f} \delta \prod_{p \leq X} \left(1 + \frac{\Delta_K(p)-1}{p}\right),
\end{align*}
where in the above, $t_0 = t_0(f_l g;X)$. Combining all of the $R-M = O_K(1)$ exceptional sets in this way, we see that for all but $O_{K,f}(\delta X)$ choices of $x \in [X,2X]$, we obtain by the triangle inequality that
$$
|\mc{S}(x)| \leq \sum_{M+1 \leq l \leq R} |c_l| \sum_{g \in \{1,\sg_f\}} \left|\sum_{x \leq n \leq x+h} f_l(n)g(n)\iota_{K,f}(n)\right| \ll_{K,f} \delta h \prod_{p \leq X} \left(1+\frac{\Delta_K(p)-1}{p}\right).
$$
Next, write
\begin{align*}
\mc{L}(x) &= \left(\sum_{0 \leq l \leq M} c_l\right) \sum_{x \leq n \leq x+h} \Delta_K(n) \iota_{K,f}(n)(1+\eta \sg_f(n)) = \left(\sum_{0 \leq l \leq M} c_l\right) \left( \mc{L}_1(x) + \eta \mc{L}_2(x)\right). 
%\\
%\tilde{\mc{L}}(x) &:= \left(\sum_{0 \leq l \leq M} c_l\right) \sum_{x \leq n \leq x+h} \Delta_K(n) \iota_{K,f}(n)(1+\eta \sg_f(n)) = \left(\sum_{0 \leq l \leq M} c_l\right) \left( \mc{L}_1(x) + \eta \mc{L}_2(x)\right)
\end{align*}
Combining Lemma \ref{lem:longSumBounds} with Corollary \ref{cor:cardCloseSums} as above, we may deduce in the same way that for all but $O_{K,f}(\delta X)$ choices of $x \in [X,2X]$ we have
$$
|\mc{L}_2(x)| \ll \delta h \prod_{p \leq X} \left(1+\frac{\Delta_K(p)-1}{p}\right).
$$
Furthermore, if we set
$$
\tilde{\mc{L}}_1(X) := \sum_{X \leq n \leq 2X} \Delta_K(n) \iota_{K,f}(n)
$$
then by Corollary \ref{cor:cardCloseSums} we once again have (taking $t_0 = 0$ since $\Delta_K\iota_{K,f} \geq 0$) that for all but $O_{K,f}(\delta X)$ choices of $x \in [X,2X]$,
$$
|\mc{L}_1(x)| \geq \frac{h}{X} \tilde{\mc{L}}_1(X) - h\Bigg|\frac 1h \mc{L}_1(x) - \frac 1X \tilde{\mc{L}}_1(X) \Bigg| \geq \frac{h}{X} \tilde{\mc{L}}_1(X) - O_{K,f}\left(\delta h\prod_{p \leq X}\left(1+\frac{\Delta_K(p)-1}{p}\right)\right).
$$
Combining all of these facts together (and taking unions over exceptional sets), we deduce that for all but $O_{K,f}(\delta X)$ choices of $x \in [X,2X]$ we have, simultaneously for both $\eta \in \{-1,+1\}$,
\begin{align*}
&\sum_{x \leq n \leq x+h} g_K(n)\iota_{K,f}(n)(1+\eta \sg_f(n)) \\
&\geq \left(\sum_{0 \leq l \leq M} c_l\right) \cdot \frac{h}{X} \sum_{X \leq n \leq 2X} \Delta_K(n)\iota_{K,f}(n) - O\left(\delta h \prod_{p \leq X} \left(1+\frac{\Delta_K(p)-1}{p}\right)\right) \\
&\geq \left(c-O_{K,f}(\delta)\right) h \prod_{p \leq X} \left(1+\frac{\Delta_K(p)-1}{p}\right),
\end{align*}
for some $c = c(K,f) > 0$, where in the last estimate we used \eqref{eq:Bd1}. Choosing $\delta$ sufficiently small as a function of $K$ and $f$, we thus obtain 
$$
\sum_{x \leq n \leq x+h} g_K(n) \iota_{K,f}(n)(1+\eta \sg_f(n)) \gg_{K,f} h \prod_{p \leq X}\left(1+\frac{\Delta_K(p)-1}{p}\right),
$$
for $\geq (1-O_{K,f}(\delta))X \geq 3X/4$ choices of $x \in [X,2X]$. Since now $h_0$ was simply chosen sufficiently large in terms of $\delta$, we can take it of some large size $O_{K,f}(1)$, in which case $h = h_0H(\Delta_K;X) \asymp_{K,f} \delta_K(X)^{-1}$, and the claim follows.
\end{proof}

\section{Proof of Theorem \ref{thm:consecSoTS}} \label{sec:shiftedSoTS}
Fix an integer $a \neq 0$. The proof of Theorem \ref{thm:consecSoTS} relies on two propositions.

\begin{prop}\label{prop:smallrnCorr}
Let  $\delta > 0$ and suppose $X^{\e} < h \leq X^{1-\e}$. Then for all but $O(X/(\log \log X)^{1/10})$ choices of $x \in [X,2X]$ we have 
$$
\sum_{\ss{x < n \leq x+ h \\ 0 < |\lambda_f(n)| < X^{-\delta}}} r(n+a) \ll_{\delta,f} \frac{h}{(\log\log X)^{1/10}}.
$$
\end{prop}

\begin{prop}\label{prop:lowBdCorrel}
Let $X^{1/3 + \e} < h \leq X$. Then for all but $O(X/(\log \log X)^{1/10})$ choices of $x \in [X,2X]$ we have
$$
\sum_{\ss{x < n \leq x + h \\ \lambda_f(n) \neq 0}} r(n+a) \gg_f h.
$$
\end{prop}

\begin{proof}[Proof of Theorem \ref{thm:consecSoTS} assuming Propositions \ref{prop:smallrnCorr} and \ref{prop:lowBdCorrel}]
Let $h > X^{1/2+3\e}$ and let $\eta \in (0,1/4)$ be a small parameter. Assume for the sake of contradiction that 
$$
\#\{x \in [X,2X] : \exists n \in [x,x+h] \text{ with } n \in a+\mc{N} \text{ and } \lambda_f(n) < 0\}  \leq \eta X
$$
(the corresponding case where $\lambda_f(n) > 0$ on $[x,x+h]$ is completely similar). Thus, for all but $\leq \eta X$ choices of $x \in [X,2X]$, we have that $\lambda_f(n)r(n-a)$ has the same sign for all $n \in [x,x+h]$. For such an $x$, 
$$
\sum_{x < n \leq x+h} |\lambda_f(n)| r(n-a) = \sum_{x < n \leq x+h} \lambda_f(n)r(n-a) \ll \max_{X < y \leq 3X} \left|\sum_{n \leq y} \lambda_f(n)r(n-a)\right|.
$$
%By a theorem of Jutila \cite[(1.32)]{Jut}, we have
By a shifted convolution sum estimate of Ravindran \cite{Rav}, we obtain
$$
\left|\sum_{x < n \leq x+h} \lambda_f(n)r(n-a)\right| \ll_{\e} X^{1/2+\e}.
$$
On the other hand, for any small $\delta >0$ we have
\begin{align*}
&\sum_{x < n \leq x+ h} |\lambda_f(n)| r(n-a) \\
&\geq X^{-\delta} \sum_{\ss{x < n \leq x+ h \\ |\lambda_f(n)| \geq X^{-\delta}}} r(n-a) = X^{-\delta}\left(\sum_{\ss{x < n \leq x+ h \\ \lambda_f(n) \neq 0}} r(n+a) - \sum_{\ss{x < n \leq x + h \\ 0 < |\lambda_f(n)| < X^{-\delta}}} r(n-a)\right).
\end{align*}
By Propositions \ref{prop:smallrnCorr} and \ref{prop:lowBdCorrel}, respectively, we have
\begin{align*}
\sum_{\ss{x < n \leq x + h \\ 0 < |\lambda_f(n)| < X^{-\delta}}} r(n-a) \ll_{\delta,f} \frac{h}{(\log\log X)^{1/5}}, \quad \quad  
\sum_{\ss{x < n \leq x+ h \\ \lambda_f(n) \neq 0}} r(n-a) \gg_f h
\end{align*}
for all but $O(X/(\log \log X)^{1/10})$ choices of $x \in [X,2X]$. \\
It therefore follows that if $X \geq X_0(\delta,\eta)$ then for all but $\leq 2\eta X$ choices of $x \in [X,2X]$ we have
$$
h \ll_{\delta,f} X^{\delta}\sum_{x < n \leq x+h} |\lambda_f(n)|r(n-a) \ll_{\e} X^{1/2+ \delta + \e}.
$$
Choosing $\delta = \e$ sufficiently small and $X$ larger in terms of $f$ and $\e$ if necessary, we deduce that $h \ll_{\e,f} X^{1/2+2\e}$, which is a contradiction for $X$ large enough. \\
Thus, there is a constant $c > 0$ such that for $\geq cX$ choices of $x \in [X,2X]$ $\lambda_f$ changes sign on $(a+\mc{N}) \cap [x,x+h]$. By the greedy selection argument mentioned at the beginning of Section \ref{sec:MainArg}, it follows that there are $\gg X/h \gg_{\e,f} X^{1/2-3\e}$ distinct such $n$, all of which produce a sign change, and the theorem is proved.
\end{proof}

%\begin{rem}
%Note that in the above propositions it is sufficient to obtain the required estimates in \emph{almost all} intervals $[x,x+h]$ with $x \in [X,2X]$. Thus, in principle, rather than appeal to the uniform estimate of Jutila, in the proof of Theorem \ref{thm:consecSoTS} we could have used the recent variance estimate
%$$
%\frac{1}{X}\int_X^{2X} \left|\sum_{n \leq x} \lambda_f(n)\lambda_f(n+a)\right|^2 \ll X^{1+\e}
%$$
%for fixed $a \neq 0$, due to Nordentoft, Petridis and Risager \cite[Thm. 1.2]{NoPeRi}, which would have resulted in the better (contradictory) bound $h \ll_{\delta,f} X^{1/2+\delta + \e}$ for all but $o(X)$ choices of $x \in [X,2X]$. However, at present it is not obvious to the author how to improve upon the restriction $h \gg X^{5/6 + \e}$, which is a consequence of bounds for $r(n)$ in arithmetic progressions (see Section \ref{subsec:rnAP}).
%\end{rem}

\subsection{Proof of Proposition \ref{prop:smallrnCorr}}
In order to prove Proposition \ref{prop:smallrnCorr} we will deduce from the condition $0 < |\lambda_f(n)| < X^{-\delta}$ that $n$ is divisible by a \emph{large} prime power $p^k$ with a \emph{small} Fourier coefficient $|\lambda_f(p^k)|$. The following lemma plays an important role in showing that the number of multiples of all such prime powers is quite small.

\begin{lem} \label{lem:verySmallLambda}
Let $c \in (0,1/4)$ be fixed and let $2 \leq Z \leq X$. Then 
$$
\sum_{\ss{Z \leq p^\nu \leq X \\ 0 < |\lambda_f(p^\nu)| \leq (\log X)^{-c}}} \frac{1}{p^\nu} \ll_f \frac{\log\log X}{(\log X)^c} + \frac{1}{(\log (Z+\log X))^{1/5}}.
$$
Similarly, we have
$$
\sum_{\ss{Z \leq p^{\nu} \leq X \\ \lambda_f(p^{\nu}) = 0}} \frac{1}{p^{\nu}} \ll_f \frac{1}{(\log Z)^{1/5}}.
$$
%where $Y := \max\{Z, \log X\}$.
\end{lem}
\begin{proof}
By Thorner's quantitative version of the Sato-Tate theorem \cite[Thm. 1.1]{Tho}, for any $\Delta > 0$ we have
$$
|\{p^\nu \leq X : |\lambda_f(p^\nu)| < \Delta\}| = \pi(X) \left(\mu_{ST}([-\Delta,\Delta]) + O\left(\frac{\log(k \log X)}{\sqrt{\log X}}\right)\right).
$$
By Lemma \ref{lem:Ser}, we also have 
$$
|\{p^\nu \leq X : \lambda_f(p^\nu) = 0\}| \ll_f \frac{1}{(\log X)^{1/4-\e}}.
$$
Applying these two results with $\Delta = (\log X)^{-c}$, $0 < c < 1/4$, then using partial summation, for any $2 \leq Y \leq X$ we obtain
\begin{align*}
\sum_{\ss{Y < p^\nu \leq X \\ 0 < |\lambda_f(p^\nu)| < (\log X)^{-c}}} \frac{1}{p^\nu} &= \mu_{ST}([-\Delta,\Delta]) \log\left(\log X/\log Y\right) + O\left(\int_Y^X \frac{du}{u(\log u)^{6/5}}\right) \\
&= \mu_{ST}([-\Delta,\Delta]) \log\left(\log X/\log Y\right) + O\left(\frac{1}{(\log Y)^{1/5}}\right) \\
&\ll \frac{\log\log X}{(\log X)^{c}} + \frac{1}{(\log Y)^{1/5}}.
\end{align*}
Now, it is known that there is a constant $C_f > 0$ such that if $\lambda_f(p^\nu) \neq 0$ then $|\lambda_f(p^\nu)| \geq p^{-C_f\nu}$. Indeed, as argued in \cite{LuRaSh}, writing $a_f(n) := \lambda_f(n)n^{(k-1)/2}$ for all $n$ we have that $a_f(n)$ is an algebraic integer in some number field $K_f$ of (finite) degree $d_f \geq 1$, and so for each $n = p^\nu$ with $a_f(n) \neq 0$ its norm $N_{K_f}(a_f(p^\nu)) \geq 1$. On the other hand, for each automorphism $\sg$ of $K_f$ it is well-known that $\{a_f(n)^{\sg}\}_n$ is the sequence of coefficients of an eigencusp form of weight $k$ and level $1$, and we obtain
$$
1 \leq N_{K_f}(a_f(p^\nu)) = |a_f(p^{\nu})| \prod_{\substack{\sg \in \text{Aut}(K_f) \\ \sg \neq 1}} |a_f(p^\nu)^\sg| \leq |\lambda_f(p^\nu)| d(p^{\nu})^{d_f-1} p^{\nu d_f(k-1)/2},
$$
using Deligne's bound $|a_f(p^{\nu})^{\sg}| \leq d(p^{\nu}) p^{\nu(k-1)/2}$ for each $\sg \neq 1$ in the last inequality. From this, the claim $|\lambda_f(p^\nu)| \geq p^{-C_f \nu}$ follows with any $C_f > d_f(k-1)/2$.  \\
Therefore, if $0 < |\lambda_f(p^\nu)| < (\log X)^{-c}$ then $p^{\nu} > (\log X)^{\eta}$ with $\eta := c/C_f$, and so we immediately have that
$$
\sum_{\ss{Z < p^\nu \leq X \\ 0 < |\lambda_f(p^\nu)| < (\log X)^{-c}}} \frac{1}{p^\nu} = \sum_{\ss{\max\{(\log X)^{\eta},Z\} < p^\nu \leq X \\ 0 < |\lambda_f(p^\nu)| < (\log X)^{-c}}} \frac{1}{p^\nu} \ll \frac{1}{(\log(Z + (\log X)^{\eta}))^{1/5}},
$$
and the first claim follows. \\
The second claim is proven more directly by appealing to Lemma \ref{lem:Ser} and partial summation.
\end{proof}

To leverage the bound in Lemma \ref{lem:verySmallLambda} we require the following upper bound for correlations of $r(n)$ supported on multiples of prime powers $p^k$, provided they are not too close to $h$ in size.

\begin{lem}\label{lem:NairTenBd}
Let $X^{\e} \leq h \leq X$. Then for any $p^\nu \leq h^{1-\e}$ we have
$$
\sum_{\ss{x < n \leq x+h \\ p^\nu || n}} r(n+a) \ll_{\e} \frac{h}{p^\nu}.
$$
\end{lem}
\begin{proof}
Write $r'(n) = r(n)/4$, which is multiplicative and divisor-bounded. By Shiu's theorem \cite{Shiu}, the sum in question is bounded above by
$$
4\sum_{\ss{x+ a < m \leq x+h+a \\ m \equiv a\pmod{p^{\nu}}}} r'(m) \ll_{\e} \frac{h}{p^{\nu}}\prod_{p' \leq X} \left(1+\frac{r'(p')-1}{p'}\right) \ll h/p^{\nu},
$$
as claimed.
\end{proof}

\begin{proof}[Proof of Proposition \ref{prop:smallrnCorr}]
Suppose $n \leq X$ satisfies $0 < |\lambda_f(n)| < X^{-\delta}$. 
Writing out the prime factorization of $n$ and using multiplicativity, we observe that
$$
X^{-\delta} > \prod_{p^k||n} |\lambda_f(p^k)| \geq \left(\min_{p^k||n} |\lambda_f(p^k)|\right)^{\omega(n)} > 0,
$$
where $\omega(n)$ denotes the number of distinct prime factors of $n$. Since $\omega(n) \leq C\frac{\log X}{\log\log X}$ for some $C > 0$ absolute, we deduce that if $n$ satisfies $0 < |\lambda_f(n)| < X^{-\delta}$ then there is a prime power $p^k || n$ for which 
$$
0 < |\lambda_f(p^k)| < \exp\left(-\delta \frac{\log X}{\omega(n)}\right) \leq (\log X)^{-\delta/C}.
$$
Set $c := \min\{\delta/C,1/5\}$, and call 
$$
\mc{P}_c := \{p^\nu \leq 2X : 0 < |\lambda_f(p^\nu)| < (\log X)^{-c}\}.
$$ 
Then it suffices to bound, for all but $O(X/(\log X)^{1/10})$ choices of $x \in [X,2X]$, the sum
$$
\sum_{\ss{ x < n \leq x + h \\ \exists p^\nu || n , p^\nu \in \mc{P}_c}} r(n+a).
$$
We split the set $\mc{P}_c$ into three ranges:
$$
\text{i) $p^\nu \leq h^{1-\e}$, ii) $h^{1-\e} < p^\nu \leq X^{1-\e}$ and iii) $X^{1-\e} < p^\nu \leq 2X$.}
$$
We immediately note that since $h < X^{1-\e}$ there is at most one multiple of any prime power $p^\nu$ from range iii) in any interval $(x,x+h]$, and since (by arguing as in the proof of Lemma \ref{lem:verySmallLambda}) there are $\ll X/(\log X)^{1+c'}$ such prime powers in $\mc{P}_c$, for some $c ' > 0$ depending on $c$, we may ignore the range $p^\nu > X^{1-\e}$ by excluding at most $O(X/(\log X)^{1+c'})$ intervals $[x,x+h]$. \\
Next, consider range i). Combining Lemmas \ref{lem:NairTenBd} and \ref{lem:verySmallLambda} we obtain that
$$
\sum_{\ss{x < n \leq x + h \\ \exists p^\nu \in \mc{P}_c \cap [2,h^{1-\e}] \\ p^\nu || n}} r(n+a) \ll \sum_{\ss{p^\nu \leq h^{1-\e} \\ p^\nu \in \mc{P}_c}} \sum_{\ss{x < n \leq x + h \\ p^\nu||n}} r(n+a) \ll_{\e} h\sum_{\ss{p^\nu \leq h^{1-\e} \\ 0 < |\lambda_f(p^\nu)| < (\log x)^{-c}}} \frac{1}{p^\nu} \ll_{\delta,f} \frac{h}{(\log\log X)^{1/5}}.
$$
Note that this holds for every $x \in [X,2X]$. \\
Finally, consider range ii). Observe that by the same argument as in Lemma \ref{lem:NairTenBd} (splitting into dyadic intervals $[X,2X]$ and $[2X,4X]$ and taking $h(y) = y$ for $y \in \{X,2X\}$),
$$
\sum_{\ss{X < n \leq 4X \\ p^\nu||n}} r(n+a) \ll \frac{X}{p^\nu}.
$$
Thus, on combining this with Lemma \ref{lem:verySmallLambda} we obtain
$$
\sum_{\ss{X < n \leq 4X \\ \exists p^\nu \in \mc{P}_c \cap (h^{1-\e} ,X^{1-\e}] \\ p^\nu||n}} r(n+a) \ll_{\e} \frac{X}{(\log \log X)^{1/5}},
$$
in light of the condition $h \geq X^{\e}$. 
By positivity, we deduce from this that
$$
\sum_{X < m \leq 2X} \sum_{\ss{m < n \leq m+h \\ \exists p^\nu \in (h^{1-\e}, X^{1-\e}] \cap \mc{P}_c \\ p^\nu||n }} r(n+a) \leq h \sum_{\ss{X < n \leq 2X+h \\ \exists p^\nu \in (h^{1-\e},X^{1-\e}] \cap \mc{P}_c \\  p^\nu||n}} r(n+a) \ll \frac{hX}{(\log \log X)^{1/5}}.
$$
By Markov's inequality, we thus obtain that for all but $O(X/(\log \log X)^{1/10})$ we have
$$
\sum_{\ss{m < n \leq m+h \\ \exists p^\nu \in (h^{1-\e}, X^{1-\e}] \cap \mc{P}_c \\ p^\nu||n}} r(n+a) \ll \frac{h}{(\log \log X)^{1/10}}.
$$
Combined with the bound for the primes $p^\nu \leq h^{1-\e}$ and $p^{\nu} > X^{1-\e}$, the claim follows.
\end{proof}

\subsection{Proof of Proposition \ref{prop:lowBdCorrel}} \label{subsec:rnAP}
%\subsection{An estimate for sums of two squares in arithmetic progressions to squarefree moduli}
To prove Proposition \ref{prop:lowBdCorrel} we sieve out short intervals $[x,x+h]$ by the sparse set of prime powers $p^\nu$ satisfying $\lambda_f(p^{\nu}) = 0$. To accomplish this we make use of an estimate for $r(n)$ in arithmetic progressions $ n\leq y$, $n \equiv a \pmod{q}$ with $q \leq y^{1/10}$, say. Asymptotic formulae of this kind, with power saving error term for the fairly wide range $q \leq y^{2/3-\e}$, were given by R.A. Smith \cite[Thm. 9]{RASm}. The following improvement of these results is due to Tolev \cite{Tol}.  
%In the sequel, for an odd modulus $q$ and $a \in (\mb{Z}/q\mb{Z})^{\times}$ let
%$$
%R(y;q,a) := \sum_{\ss{n \leq y \\ n \equiv a \pmod{q}}} r(n) - \pi \prod_{p|q} \left(1-\frac{\chi_{-4}(p)}{p}\right) \frac{y}{q}.
%$$
%In this case, the error term in Tolev's estimate is, uniformly over $a$,
%$$
%R(y;q,a) \ll_{\e} y^{\e}\left(q^{1/2} + y^{1/3}\right).
%$$
%As is consistent with the well-known conjecture in the circle problem (when $q = 1$), we expect that $y^{1/3}$ here could be replaced by $y^{1/4}$. In the case at hand, where we will need only treat odd, squarefree moduli $q$, we may achieve this better bound, at least for $q$ of size a bit larger than $y^{1/2}$, by modifying a method due to Varbanets \cite{Var}. Precisely, Varbanets showed that for any $y^{1/2} < q \leq y^{2/3}$, taking $a = 1$ specifically one can obtain the error term
%$$
%R(y;q,1) \ll y^{\e}\left(q^{1/2} + y^{1/2}q^{-1/2}\right).
%$$
%Tolev's result is as follows.
%\begin{prop}\label{prop:VarbThm}
%Let $X \geq 1$ be large, let $q$ be an odd squarefree that satisfies $X^{1/2} < q \leq X^{2/3}$. Then for any $a \in (\mb{Z}/q\mb{Z})^{\times}$,
%$$
%\sum_{\ss{n \leq X \\ n \equiv a \pmod{q}}} r(n) = \frac{\pi X}{q}\prod_{p|q}\left(1-\frac{\chi_4(p)}{p}\right) + O_{\e}\left(X^{\e}\left(q^{1/2} + x^{1/2}q^{-1/4}\right)\right).
%$$
%\end{prop}
%\sum_{\ss{n \leq y \\ n \equiv 1 \pmod{q}}} r(n) =  \pi \gamma_0 \prod_{p|q} \left(1-\frac{\chi_{-4}(p)}{p}\right) \frac{y}{q} + O\left(y^{\e}
\begin{lem}[Tolev \cite{Tol}] \label{lem:SmithThm}
Let $y$ be large, $q \leq y^{2/3}$ be squarefree and let $a$ be a residue class modulo $q$ with $(a,q) = 1$. Then
$$
\sum_{\ss{n \leq y \\ n \equiv a \pmod{q}}} r(n) = \pi \prod_{p  | q} \left(1-\frac{\chi_{4}(p)}{p}\right) \frac{y}{q} + O\left(y^{\e}\left(q^{1/2} + y^{1/3}\right)\right),
%+ O\left(X^{2/3+\e} q^{-1/2} \right).
$$
where $\chi_4$ is the non-principal character mod $4$.
\end{lem}

\begin{proof}[Proof of Proposition \ref{prop:lowBdCorrel}]
Let $z := (\log X)^{1/2}$ and let 
$$
P = P(z) := \prod_{\ss{3 \leq p \leq z \\ \exists \nu \geq 1 \text{ with } \lambda_f(p^\nu) = 0}} p.
$$ 
By the prime number theorem, we have $P \leq \exp(2z) \ll_{\e} X^{\e}$.  \\
%We may lower bound the sum in question by
We may lower bound the sum in question by
\begin{align}\label{eq:lowBdforProp}
\geq \sum_{\ss{x < n \leq x + h \\ 2 \nmid n \\ \lambda_f(n) \neq 0}} r(n+a) \geq \sum_{\ss{x < n \leq x + h \\ 2 \nmid n \\ (n,P(z)) = 1}} r(n+a) - \sum_{\ss{x < n \leq x + h \\ \exists p > z, k \geq 1 \\ \lambda_f(p^k) = 0 \\ p^k||n }} r(n+a).
\end{align}
To estimate the first sum we use M\"{o}bius inversion together with Lemma \ref{lem:SmithThm}, getting
\begin{align*}
\sum_{\ss{x < n \leq x + h \\ 2 \nmid n \\ (n,P(z)) = 1}} r(n+a)
&= \sum_{\substack{d|P}} \mu(d) \sum_{\ss{x<n \leq x+ h \\ 2 \nmid n \\ d|n}} r(n+a) \\
&= \sum_{e|(a,P)} \mu(e) r'(e) \sum_{\substack{d'|P/e \\ (d',a/e) = 1}} \mu(d') \sum_{\substack{x/e < m \leq (x+h)/e \\ m \equiv 1 \pmod{2} \\ m \equiv 0 \pmod{d'}}} r(m+a/e)  \\ 
&= \frac{\pi h}{2} \sum_{e|(a,P)} \frac{\mu(e)r'(e)}{e} \sum_{\substack{d'|P/e \\ (d',a/e) = 1}} \frac{\mu(d')}{d'} \prod_{p|d'} \left(1-\frac{\chi_4(p)}{p}\right)  + O\left(X^{\e} \left(\sum_{d|P} d^{\tfrac 12} + X^{\tfrac 13+\e} \tau(P) \right)\right) \\
&= \frac{\pi h}{2} \prod_{p|(a,P)} \left(1-\frac{r'(p)}{p-1+\chi_4(p)/p}\right) \prod_{\substack{3 \leq p\leq z \\ \lambda_f(p) = 0}} \left(1-\frac{1+\chi_4(p)}{p} + \frac{\chi_4(p)}{p^2}\right) + O\left(X^{1/3+\e} \right).
%&= \pi h \prod_{\ss{p \leq z \\ \lambda_f(p) = 0}} \left(1-\frac{1}{p}+ \frac{\chi_4(p)}{p^2}\right) + O(X^{1/3+\e}).
%\left(\frac{1}{2}(1+1_{2\nmid e}1_{\lambda_f(2^j) \neq 0 \forall j}) \prod_{\ss{p \geq 5 : \exists k \geq 1 \\ \lambda_f(p^k) = 0 \\ p \equiv 1 \pmod{4}}} \left(1-\frac{(2p+1)(p+1)}{p^2(p-1)}\right)\left(1-\frac{2}{p}\right) + o(1)\right) h + O\left(X^{5/6+\e}\right),
\end{align*}
%where in the error term we used the trivial bound $D = D(X) \ll_{\e} X^{\e}$. \\
Note that as $P$ is odd, $2\nmid (a,P)$ and thus the products above are $> 0$. Furthermore, since $\sum_{p  : \lambda_f(p) = 0} \frac{1}{p} < \infty$, for $h > X^{1/3+2\e}$ the product over $3 \leq p \leq z$ with $\lambda_f(p) = 0$ converges, and thus we deduce the existence of a constant $c_{f,a} > 0$ such that
$$
\sum_{\ss{x < n \leq x+h \\ 2 \nmid n \\ (n,P(z)) = 1}} r(n+a) = (c_{f,a}+o(1)) h.
$$
Next, we upper bound the second sum in \eqref{eq:lowBdforProp} similarly as in Proposition \ref{prop:smallrnCorr}. For $p^k \leq h^{1-\e}$ we use Lemmas \ref{lem:verySmallLambda} and \ref{lem:NairTenBd} to get
$$
\sum_{\ss{x < n \leq x + h \\ \exists p > z, \nu \geq 1 \\ p^\nu \leq h^{1-\e} \\ \lambda_f(p^\nu) = 0 \\ p^\nu||n }} r(n+a) \ll h\sum_{\ss{z < p^\nu \leq h^{1-\e} \\ \lambda_f(p^\nu) = 0}} \frac{1}{p^\nu} \ll_f \frac{h}{(\log z)^{1/5}} \ll \frac{h}{(\log \log X)^{1/5}},
$$
and, if $x\in [X,2X]$ is chosen outside of a set of size $\ll X/(\log \log X)^{1/10}$ then we also have
$$
\sum_{\ss{x < n \leq x + h \\ \exists p > z, \nu \geq 1 \\ p^\nu > h^{1-\e} \\ \lambda_f(p^\nu) = 0,  p^\nu||n }} r(n+a) \ll_f \frac{h}{(\log \log X)^{1/10}}.
$$
Adding up the two contributions, we get that
$$
\sum_{\ss{x < n \leq x + h \\ \exists p > z, \nu \geq 1 \\ \lambda_f(p^\nu) = 0,  p^\nu||n }} r(n+a) \ll_f \frac{h}{(\log\log X)^{1/10}}.
$$
To summarize, we have shown the existence of a constant $c_f > 0$ such that if $h > X^{1/3+2\e}$ then
$$
\sum_{\ss{ x < n \leq x+ h \\ \lambda_f(n) \neq 0}} r(n+a) \geq (c_{f,a}+o(1)) h + O_f\left(\frac{h}{(\log\log X)^{1/10}} \right) \gg_f h
$$
for all but $O(X/(\log \log X)^{1/10})$ choices of $x \in [X,2X]$, and the claim follows.
\end{proof}

\bibliographystyle{plain}
\bibliography{FourierCoeffMF.bib}

\end{document}